\documentclass{article}
\usepackage{amsmath,amssymb}
\DeclareOldFontCommand{\sc}{\normalfont\scshape}{\mathrm}
\renewcommand{\Bbb}{\mathbb}
\newtheorem{thm}{Theorem}[section]
\newtheorem{lem}[thm]{Lemma}
\newtheorem{prop}[thm]{Proposition}
\newtheorem{cor}[thm]{Corollary}
\newtheorem{defn}[thm]{Definition}
\newenvironment{defn-new}{\begin{defn} \em}{\end{defn}}
\newtheorem{rem}[thm]{Remark}
\newenvironment{rem-new}{\begin{rem} \em}{\end{rem}}
\newtheorem{ex}[thm]{Example}
\newenvironment{ex-new}{\begin{ex} \em}{\end{ex}}
\newtheorem{prob}[thm]{Problem}
\newenvironment{prob-new}{\begin{prob} \em}{\end{prob}}

\newenvironment{notation-new}{\begin{rem} \em}{\end{rem}}
\newtheorem{agr}[thm]{Agreement}
\newenvironment{agr-new}{\begin{agr} \em}{\end{agr}}

\makeatletter \@addtoreset{equation}{section} \makeatother

\begin{document}

\begin{center}
{\Large {\bf {Newman--Penrose formalism in $3$-dimensional trans-Sasakian
manifolds}}}
\end{center}

\begin{center}
{\large Prachi$^{1}$, Marie-Am\'{e}lie Lawn$^{2}$ and Mukut Mani
Tripathi$^{3}$}
\end{center}

\begin{center}
$^{1}$Department of Mathematics,\\
Institute of Science, \\
Banaras Hindu University,\\
Varanasi 221005, India\\
E-mail: prachisinghsvc203131@gmail.com\\
ORCID: 0009-0006-0218-8109

\vspace{0.3cm}

$^{2}$Department of Mathematics,\\
Faculty of Natural Sciences, \\
Imperial College London, London,\\
180 Queen's Gate, London SW72AZ, UK\\
E-mail: m.lawn@imperial.ac.uk\\
ORCID: 0000-0003-2691-0195

\vspace{0.3cm}

$^{3}$Department of Mathematics,\\
Institute of Science, \\
Banaras Hindu University,\\
Varanasi 221005, India\\
E-mail: mmtripathi66@yahoo.com\\
ORCID: 0000-0002-6113-039X
\end{center}

\bigskip

\begin{quote}
{\bf Abstract.} We study $3$-dimensional trans-Sasakian manifolds using the Newman--Penrose formalism. In this framework, the geometry of the structure vector field is encoded by scalar spin coefficients: acceleration, shear, expansion, and twist. A central observation is that, in dimension $3$, the trans-Sasakian condition is equivalent to the characteristic vector field defining a shear-free geodesic congruence, or equivalently a conformal foliation by geodesics. Thus, the Newman--Penrose equations provide a direct scalar formulation of the conformal foliations studied by Baird and Wood in the theory of harmonic morphisms. Within this framework, we derive curvature and Laplacian identities for trans-Sasakian manifolds and their main subclasses, including formulae for the Ricci tensor, scalar curvature, Einstein condition, rough Laplacian, divergence and harmonicity of the characteristic vector field, together with several illustrative examples. As an application, we consider trans-Sasakian structures compatible with fixed homogeneous metrics of type ${\Bbb E}(\kappa,\tau)$. We prove a rigidity result: in the non-space-form cases, the Newman--Penrose equations force the characteristic vector field to be vertical. Hence, for $\tau\neq0$ and $\kappa\neq4\tau^2$, every compatible trans-Sasakian structure is the canonical vertical $\alpha$-Sasakian structure, while for $\tau=0$ and $\kappa\neq0$, it is vertical and cosymplectic. In particular, these non-space-form homogeneous metrics admit no proper compatible trans-Sasakian structures.
\end{quote}

\bigskip


\bigskip

\noindent {\em $2020$ Mathematics Subject Classification\/}{\rm :}
Primary 53C25; Secondary 53C21, 83C60

\bigskip


\section{Introduction}

The Newman--Penrose formalism, introduced by Ezra T. Newman and Roger Penrose
\cite[Newman and Penrose 1962]{Newman-Penrose-62-JMP}, provides a powerful
framework in general relativity for studying spacetime geometry through a
complex null tetrad. Its essential feature is the replacement of the tensorial
description of the Levi--Civita connection by complex scalar invariants, the
spin coefficients, which encode the geometry of the chosen frame. In subsequent
work \cite[Newman and Penrose 1965]{Newman-Penrose-65-PRL}, Newman and Penrose
derived five complex, equivalently ten real, constants of motion for Bondi-type
solutions of the Einstein field equations. This programme was further extended
in \cite[Exton {\it et al.} 1969]{Exton-Newman-Penrose-69-JMP} to
Einstein--Maxwell spacetimes.

A refinement of the Newman--Penrose formalism was later developed by Geroch
{\it et al.} \cite[Geroch {\it et al.} 1973]{Geroch-Held-Penrose-73-JMP}.
Their approach keeps only a pair of null directions, together with spin- and
boost-weighted quantities and differential operators preserving these weights.
This gives a more invariant version of the spin-coefficient formalism, lying
between the original Newman--Penrose formalism and a fully covariant tensorial
description.

The Newman--Penrose formalism has also been adapted to Riemannian
three-manifolds. Aazami \cite[Aazami 2015]{Aazami-15-JGP} introduced the
formalism in this setting, and Matsuno \cite[Matsuno 2025]{Matsuno-25-arXiv}
recently applied it to $3$-dimensional almost contact metric manifolds,
characterizing several standard classes, including contact, Sasakian,
cosymplectic, Kenmotsu, trans-Sasakian, and $(\kappa,\mu)$-contact metric
manifolds. In the original Lorentzian setting, the Newman--Penrose formalism is
particularly well suited to the study of shear-free geodesic congruences. In
dimension three, the same language is natural for almost contact metric
geometry: normal almost contact metric structures are characterized by the
property that the structure vector field generates a shear-free geodesic
congruence. Thus trans-Sasakian geometry can be studied effectively through the
twist, expansion, and shear of this congruence.

It is well known that every orientable Riemannian $3$-manifold admits almost
contact metric structures compatible with the given metric; see
\cite[Cabrerizo {\em et al.} 2009]{Cab-FG-09}. Hence, once $g$ and an orientation are fixed, compatible
almost contact metric structures are abundant. The geometry becomes restrictive
only after imposing additional conditions, such as normality, the contact
metric condition, the Sasakian or trans-Sasakian equations, homogeneity, or
compatibility with a prescribed fibration.

In this paper we study the trans-Sasakian condition for the homogeneous metrics
of the spaces ${\Bbb E}(\kappa,\tau)$. Thurston's eight simply connected model
geometries in dimension $3$ are (see \cite[Thurston 1997, pp.\ 179-190]{Thurston-97})
\[
{\Bbb R}^3,\qquad {\Bbb S}^3,\qquad {\Bbb H}^3,\qquad
{\Bbb S}^2\times{\Bbb R},\qquad {\Bbb H}^2\times{\Bbb R},\qquad
{\rm Nil}_3,\qquad
\widetilde{{\rm SL}}(2,{\Bbb R}),\qquad
{\rm Sol}_3 .
\]
Here ${\Bbb R}^3$, ${\Bbb S}^3$, and ${\Bbb H}^3$ are the space-form
geometries; ${\Bbb S}^2\times{\Bbb R}$ and ${\Bbb H}^2\times{\Bbb R}$ are
product geometries; ${\rm Nil}_3$ and $\widetilde{{\rm SL}}(2,{\Bbb R})$ are
non-trivial homogeneous fibrations over surfaces; and ${\rm Sol}_3$ is the
remaining solvable geometry. The family ${\Bbb E}(\kappa,\tau)$ contains the
product geometries when $\tau=0$, and for $\tau\neq0$ includes ${\rm Nil}_3$,
Berger-sphere metrics, and $\widetilde{{\rm SL}}(2,{\Bbb R})$, with the
spherical space-form case corresponding to the exceptional values
$\kappa=4\tau^2$, and the flat case to $\kappa=\tau=0$.

By a trans-Sasakian structure compatible with a metric of type ${\Bbb
E}(\kappa,\tau)$, we mean an almost contact metric structure
$(\varphi,\xi,\eta,g)$ satisfying the trans-Sasakian equation, where the metric
$g$ is fixed to be the prescribed homogeneous metric. This is not the same as
classifying homogeneous almost contact metric or trans-Sasakian structures,
where $(\varphi,\xi,\eta,g)$ itself is assumed to be invariant under a chosen
transitive group action. Here no homogeneity assumption is imposed on
$\varphi$, $\xi$, $\eta$, or on the functions $\alpha$ and $\beta$. The problem
is therefore a fixed-metric rigidity problem: given a homogeneous metric,
determine whether any compatible trans-Sasakian structure can occur,
homogeneous or not.

Our main rigidity result is that, in the non-space-form ${\Bbb E}(\kappa,\tau)$
cases, the trans-Sasakian equations force the structure vector field to be the
canonical vertical field. More precisely, if $\tau\neq0$ and
$\kappa\neq4\tau^2$, then every compatible trans-Sasakian structure is vertical
and hence reduces to the canonical $\alpha$-Sasakian case. If $\tau=0$ and
$\kappa\neq0$, corresponding to $M^3=M^2_\kappa\times{\Bbb R}$, then every
compatible trans-Sasakian structure is again vertical and hence cosymplectic.
In particular, no proper trans-Sasakian structures are compatible with these
fixed homogeneous metrics.

A second important theme of the paper is the precise equivalence, in dimension
three, between compatible trans-Sasakian structures, shear-free geodesic
congruences, and conformal foliations by geodesics. We express this equivalence
in Newman--Penrose spin coefficients and thereby relate trans-Sasakian geometry
to the conformal foliations considered by Baird and Wood in the theory of
harmonic morphisms. From this viewpoint, the trans-Sasakian equation becomes a
scalar system governing the congruence generated by the structure vector field.
This equivalence gives a practical advantage beyond terminology.  The
Baird--Wood theory studies conformal foliations by geodesics through harmonic
morphisms and quotient foliations, whereas the Newman--Penrose formalism
detects the same condition directly through spin coefficients: in dimension
three a unit vector field generates a conformal foliation by geodesics if and
only if
\[
\kappa_{\mathrm{np}}=0,\qquad \sigma_{\mathrm{np}}=0.
\]
Thus the existence problem becomes a system of local scalar differential
equations attached to a congruence.  For the homogeneous metrics ${\Bbb
E}(\kappa,\tau)$, this makes the rigidity mechanism especially transparent:
starting from an arbitrary compatible structure, the Newman--Penrose equations
themselves force the structure vector field to be vertical. We expect the same
approach to be useful for other fixed $3$-dimensional metrics, such as
left-invariant metrics, cohomogeneity-one metrics and warped product
geometries, where the connection coefficients are explicit and the equations
$\kappa_{\mathrm{np}}=\sigma_{\mathrm{np}}=0$ give a direct test for conformal
geodesic foliations.

The paper is organized as follows. In Section~\ref{Sect-Prel}, we review
trans-Sasakian geometry and the Newman--Penrose formalism. In
Section~\ref{Sect-tsm-equiva-NP}, we prove the equivalence between
trans-Sasakian structures, shear-free geodesic congruences, and conformal
foliations by geodesics, and formulate it in Newman--Penrose notation. In
Section~\ref{Sect-NP-cur}, we derive curvature formulae in the Newman--Penrose
framework, including expressions for the Ricci tensor, the scalar curvature,
the Riemann curvature tensor, and the Einstein condition for trans-Sasakian
manifolds and their main subclasses, together with several illustrative examples. In Section~\ref{Sect-NP-RL}, we study the
rough Laplacian and divergence of the structure vector field, including the
conditions for parallelity, harmonicity, pointwise collinearity with the rough
Laplacian, and divergence-freeness. Finally, in Section~\ref{Sect-NP-homo-met},
we apply the formalism to the homogeneous metrics ${\Bbb E}(\kappa,\tau)$ and
prove the rigidity results described above.

\section{Preliminaries\label{Sect-Prel}}

\subsection{Trans-Sasakian manifolds}

Let $M$ be a $(2n+1)$-dimensional smooth manifold. An almost contact metric
structure $(\varphi,\xi,\eta,g)$ on $M$ consists of a $(1,1)$-tensor field
$\varphi$, a vector field $\xi$, a $1$-form $\eta$ and a compatible Riemannian
metric $g$ satisfying
\[
{\varphi}^{2} =  -I+\eta \otimes \xi, \quad \eta(\xi) = 1,\quad \varphi \xi =
0, \quad \eta \circ \varphi = 0,
\]
\[
g\left(\varphi {X},\varphi {Y}\right) = g\left(X,Y\right) - \eta \left(X\right)
\eta \left(Y\right)
\]
for all vector fields $X$, $Y$ on $M$. Then $(M,\varphi,\xi,\eta,g)$ is called
an almost contact metric manifold. For details, we refer to \cite[Blair
2010]{Blair-10-Book}.

An almost contact metric structure $(\varphi,\xi,\eta,g)$ on a manifold $M$ is
called a trans-Sasakian structure \cite[Oubina 1985]{Oubina-85-PMD} if the
product manifold $(M\times {\Bbb R},J,G)$ belongs to the class ${\cal W}_{4}$
defined in \cite[Gray and Hervella 1980]{Gray-Hervella-80-AMPA}, where $J$ is
the almost complex structure on $M\times {\Bbb R}$ given by
\[
J\left(X,f\frac{d}{dt}\right) = \left(\varphi{X}-f\xi,\eta (X)\frac{d}{dt}\right)
\]
for every vector field $X$ on $M$ and every smooth function $f$ on
$M\times{\Bbb R}$, and $G$ denotes the product metric on $M\times {\Bbb R}$. It
is well-known that an almost contact metric structure $(\varphi,\xi,\eta,g)$ on
a manifold $M$ is a trans-Sasakian structure if and only if \cite[Blair and
Oubina 1990]{Blair-Oubina-90-PM}
\begin{equation}
\left(\nabla_{X} \varphi \right) Y=\alpha \left(g\left(X,Y\right) \xi -
\eta\left(Y\right) X\right) + \beta\left(g\left(\varphi{X},{Y}\right)\xi -
\eta\left(Y\right) \varphi{X}\right)  \label{eq-tr-sas-def-a}
\end{equation}
for some smooth functions $\alpha$ and $\beta$ on $M$, where $\nabla$ denotes
the Levi-Civita connection of $g$. An almost contact metric structure
$(\varphi,\xi,\eta,g)$ satisfying (\ref{eq-tr-sas-def-a}) is called a
trans-Sasakian structure of type $(\alpha,\beta)$ \cite[Blair and Oubina
1990]{Blair-Oubina-90-PM}.

An $(\alpha,\beta)$ trans-Sasakian structure reduces to
\begin{itemize}
\item[{\bf (a)}] a ${\cal C}_{6}$-structure (\cite[Chinea and Gonzalez
1990]{Chinea-Gonzalez-90-AMPA}, \cite[Marrero 1992]{Marrero-92-AMPA}) if
$\beta=0$,

\item[{\bf (b)}] an $\alpha$-Sasakian structure \cite[Janssens and Vanhecke
1981, Theorem~2.6]{Vanh-Jan-81-Kodai} if $\beta=0\neq \alpha =$ constant,

\item[{\bf (c)}] a Sasakian structure (\cite[Sasaki and Hatakeyama
1962]{Sasaki-Hatakeyama-62-JMSJ}, \cite[Blair 2010, p.~86]{Blair-10-Book}) if
$\alpha=1$, $\beta=0$,

\item[{\bf (d)}] a ${\cal C}_{5}$-structure (\cite[Chinea and Gonzalez
1990]{Chinea-Gonzalez-90-AMPA}, \cite[Marrero 1992]{Marrero-92-AMPA}) if
$\alpha=0$,

\item[{\bf (e)}] a $\beta$-Kenmotsu structure \cite[Janssens and Vanhecke
1981, Theorem~2.8]{Vanh-Jan-81-Kodai} if $\alpha=0\neq \beta =$ constant,

\item[{\bf (f)}] a Kenmotsu structure \cite[Kenmotsu 1972]{Kenmotsu-72-Tohoku}
if $\alpha=0$, $\beta=1$,

\item[{\bf (g)}] a cosymplectic structure \cite[Blair 1967]{Blair-67-JDG} if
$\alpha =0=\beta$.
\end{itemize}

A trans-Sasakian manifold $(M,\varphi,\xi,\eta,g)$ satisfies the following
properties \cite[De and Tripathi 2003]{DeUC-Tri-03-Kyung}:
\begin{eqnarray}
R\left(X,Y\right) \xi &=&\left(\alpha^{2} - \beta^{2}\right)
\left(\eta \left(Y\right) X -\eta \left(X\right) Y\right)
+2\alpha \beta \left(\eta \left(Y\right) \varphi {X}-\eta \left(X\right)
\varphi {Y}\right)  \nonumber \\ &&+{\rm d}\alpha \left(Y\right) \varphi {X}
-{\rm d}\alpha \left(X\right)\varphi {Y} + {\rm d}\beta \left(Y\right)
{\varphi}^{2}X-{\rm d}\beta \left(X\right) {\varphi}^{2}Y, \label{eq-tr-sasa-cur-ope}
\end{eqnarray}
\begin{equation}
S\left(X,\xi \right) = \left(2n\left(\alpha^{2}- \beta^{2}\right) -{\rm d}\beta
\left(\xi \right) \right) \eta\left(X\right) -\left(2n-1\right) {\rm d}\beta
\left(X\right) -{\rm d}\alpha \left(\varphi X\right),  \label{eq-tsm-S(X,xi)}
\end{equation}
\begin{equation}
2\alpha \beta +{\rm d}\alpha  \left(\xi \right)=0,
\label{eq-condition-alpha-beta-xi}
\end{equation}
\begin{equation}
\nabla_{X}\xi =-\alpha \varphi {X}+\beta \left(X-\eta\left(X\right) \xi\right),
\label{eq-nabla_X(xi)}
\end{equation}
\begin{equation}
\left(\nabla_{X}\eta \right) Y = \alpha g\left(X,\varphi {Y}\right)
+\beta\left(g\left(X,Y\right) - \eta \left(X\right) \eta \left(Y\right) \right)
\label{eq-nabla_X(eta)}
\end{equation}
for all vector fields $X$, $Y$ on $M$, where $R$ denotes the Riemann curvature
tensor and $S$ denotes the Ricci tensor. It is known that any trans-Sasakian
manifold of dimension $\geq 5$ must be either cosymplectic, or an
$\alpha$-Sasakian or a $\beta$-Kenmotsu manifold \cite[Marrero
1992]{Marrero-92-AMPA}. Consequently, proper trans-Sasakian manifolds are
always of dimension $3$.

\subsection{Newman--Penrose formalism}

In this subsection, we briefly review the Newman--Penrose formalism adapted to
$3$-dimensional Riemannian manifolds, and follow the terminology and definition
introduced in \cite[Aazami 2015]{Aazami-15-JGP} and \cite[Matsuno
2025]{Matsuno-25-arXiv}.

Let $(M,g)$ be an orientable $3$-dimensional Riemannian manifold. Since every
orientable $3$-dimensional Riemannian manifold is parallelizable, one can
choose a global orthonormal frame $(e_{1},e_{2},\xi)$. By introducing the
complex vector fields
\begin{equation}
\partial =\frac{1}{\sqrt{2}}(e_{1}-ie_{2}), \qquad
\overline{\partial}=\frac{1}{\sqrt{2}}(e_{1}+ie_{2}),
\label{eq-partial-partial-bar}
\end{equation}
which are orthogonal to $\xi$, we obtain a complex frame
$(\partial,\overline{\partial},\xi)$. This complex frame will be called the
Newman--Penrose frame in dimension $3$. With respect to this frame, the
Riemannian metric $g$ satisfies
\begin{equation}
g(\xi,\xi) = g(\partial,\overline{\partial}) = 1, \qquad
g(\xi,\partial)=g(\partial,\partial)=g(\overline{\partial},\overline{\partial})
= 0. \label{eq-NPF-metric}
\end{equation}
Let $(\theta^{1},\theta^{2},\eta)$ denote the dual to the orthonormal frame
$(e_{1},e_{2},\xi)$. Defining the complex $1$-forms
\begin{equation}
\wp = \frac{1}{\sqrt{2}}(\theta^{1}+i\theta^{2}),\qquad
\overline{\wp}=\frac{1}{\sqrt{2}}(\theta^{1}-i\theta^{2}), \label{eq-mu-mu-bar}
\end{equation}
it follows that the complex dual frame corresponding to the complex frame
$(\partial,\overline{\partial},\xi)$ is given by $(\wp,\overline{\wp},\eta)$,
so that $\wp(\partial)=1$ and
\begin{equation}
\eta \left(\xi \right) = 1, \qquad \eta \left(\partial \right)=\eta
\left(\overline{\partial}\right) =0.  \label{eq-eta-xi-partial}
\end{equation}
With respect to the frame $(\partial,\overline{\partial},\xi)$, the spin
coefficients are defined as follows:
\begin{equation}
\sigma_{\sc np} = -g(\partial,\nabla_{\partial}\xi), \quad \rho_{\sc np} =
g(\partial,\nabla_{\overline{\partial}}\xi),\quad \kappa_{\sc np}
=-g(\partial,\nabla_{\xi}\xi), \label{eq-spin-coe-1}
\end{equation}
\begin{equation}
\beta_{\sc np} =g(\overline{\partial},\nabla_{\partial}\partial),
\quad\epsilon_{\sc np} =g(\overline{\partial},\nabla_{\xi}\partial).
\label{eq-spin-coe-2}
\end{equation}
The coefficient $\sigma_{\sc np}$ is referred to as the complex shear of $\xi$.
If $\kappa_{\sc np}=0$, then the integral curves of $\xi$ are geodesic. The
coefficient $\epsilon_{\sc np}$ is purely imaginary; although one can always
choose a local frame such that $\epsilon_{\sc np}=0$, such a frame may fail to
exist globally. The coefficient $\beta_{\sc np}$ is related to the sectional
curvature of the distribution orthogonal to $\xi$. In the homogeneous case, it
is frequently possible to select a frame for which $\beta_{\sc np}=0$. We now
express
\begin{equation}
\rho_{\sc np}=\Theta_{\sc np}+i\omega_{\sc np}, \label{eq-rho}
\end{equation}
where the real part $\Theta_{\sc np}$ is called the expansion of $\xi$, and the
imaginary part $\omega_{\sc np}$ is known as the twist of $\xi$. More
precisely, for the congruence generated by $\xi$, the expansion $\Theta_{\sc
np}$ measures the rate of change of the area of a cross-section orthogonal to
$\xi$, the twist $\omega_{\sc np}$ describes the angular velocity of such a
cross-section, and the shear $\sigma_{\sc np}$ describes its shear deformation.
Equivalently, these quantities may be interpreted as the coefficients appearing
in the corresponding covariant derivatives and Lie brackets given as follows:
\begin{equation}
\nabla_{\partial}\partial=\beta_{\sc np}\partial+\sigma_{\sc np}\xi,
\label{eq-PNF-nabla(partial,partial)}
\end{equation}
\begin{equation}
\nabla_{\partial}\overline{\partial} = -\beta_{\sc np}\overline{\partial}
-\overline{\rho_{\sc np}}\xi,  \label{eq-PNF-nabla(partial,barpartial)}
\end{equation}
\begin{equation}
\nabla_{\partial}\xi = \overline{\rho_{\sc np}}\partial -\sigma_{\sc np}
\overline{\partial}, \label{eq-PNF-nabla(partial,xi)}
\end{equation}
\begin{equation}
\nabla_{\xi}\partial = \epsilon_{\sc np} \partial+\kappa_{\sc np} \xi,
\label{eq-PNF-nabla(xi,partial)}
\end{equation}
\begin{equation}
\nabla_{\xi}\xi =-\overline{\kappa_{\sc np}}\partial -\kappa_{\sc np}
\overline{\partial}, \label{eq-PNF-nabla(xi,xi)}
\end{equation}
\begin{equation}
[\partial,\overline{\partial}] = \overline{\beta_{\sc np}}\partial -\beta_{\sc
np} \overline{\partial}+(\rho_{\sc np} -\overline{\rho_{\sc np}})\xi,
\label{eq-PNF-Lie(partial),barpartial}
\end{equation}
\begin{equation}
[\partial,\xi]= -(\epsilon_{\sc np} - \overline{\rho_{\sc np}})\partial
-\sigma_{\sc np} \overline{\partial}-\kappa_{\sc np} \xi .
\label{eq-PNF-Lie(xi,partial)}
\end{equation}

The choice of the complex frame $(\partial,\overline{\partial},\xi)$ is not
unique, and both local and global gauge transformations are possible. For
simplicity, we restrict our attention to global gauge transformations, noting
that the same considerations apply in the local setting. Consider
$e^{i\theta}:M\rightarrow {\Bbb S}^{1}$ to be a smooth map. A gauge
transformation is then given by $\partial \rightarrow
\partial^{\prime}=e^{i\theta}\partial$. Under this gauge transformation, the
spin coefficients are given by
\begin{equation}
\sigma_{\sc np}^{\prime}=e^{2i\theta}\sigma_{\sc np}, \quad \rho_{\sc
np}^{\prime}=\rho_{\sc np}, \quad\kappa_{\sc
np}^{\prime}=e^{i\theta}\kappa_{\sc np},  \label{eq-NPF-gauge-spin-coe-1}
\end{equation}
\begin{equation}
\beta_{\sc np}^{\prime} = e^{i\theta}(\beta_{\sc np} +i\partial
\theta),\quad\epsilon_{\sc np}^{\prime}=\epsilon_{\sc np} +i\xi (\theta).
\label{eq-NPF-gauge-spin-coe-2}
\end{equation}
It follows from the above transformation, the value of $\rho_{\sc np}$ is
well-defined. In contrast, for $\kappa_{\sc np}$ and $\sigma_{\sc np}$, only
their absolute values $|\kappa_{\sc np} |$ and $|\sigma_{\sc np} |$, are
well-defined. A quantity $q$ is said to have spin weight $s$ if it transforms
as $q^{\prime}=e^{is\theta}q$ under a gauge transformation. The spin weights of
$\kappa_{\sc np}$, $\rho_{\sc np}$, and $\sigma_{\sc np}$ are $1$, $0$, and
$2$, respectively.

In general, we must be careful when differentiating spin-weighted quantities,
because the derivatives $\partial(q)$, $\overline{\partial}(q)$, and $\xi(q)$
do not necessarily remain spin-weighted. However, it is possible to define
suitable operators that preserve the spin weight, as described below
\cite[Geroch {\em et al.} 1973]{Geroch-Held-Penrose-73-JMP}:
\begin{equation}
\eth (q)=\partial (q)-sq\beta_{\sc np},  \label{eq-spin-weight2}
\end{equation}
\begin{equation}
\overline{\eth}(q)=\overline{\partial}(q)+sq\overline{\beta_{\sc np}},
\label{eq-spin-weight3}
\end{equation}
\begin{equation}
P(q)=\xi (q)-s\epsilon_{\sc np} q,  \label{eq-spin-weight1}
\end{equation}
where $s$ is the spin weight of $q$. The spin-weighted quantities $\eth(q)$,
$\overline{\eth}(q)$, and $P(q)$ are of spin weights $s+1$, $s-1$ and $s$,
respectively \cite[Matsuno 2025]{Matsuno-25-arXiv}.

The spin coefficients are not mutually independent. They satisfy certain
compatibility relations, known as the generalized Sachs equations, which are
given as follows \cite[Aazami 2015]{Aazami-15-JGP}:
\begin{equation}
\xi (\sigma_{\sc np})-\partial (\kappa_{\sc np}) = \kappa_{\sc
np}^{2}+2\sigma_{\sc np} \epsilon_{\sc np} -\sigma_{\sc np} (\rho_{\sc np}
+\overline{\rho_{\sc np}})-\kappa_{\sc np} \beta_{\sc np}
+S(\partial,\partial),  \label{eq-gen-Sach-2}
\end{equation}
\begin{equation}
-\partial (\rho_{\sc np}) - \overline{\partial}(\sigma_{\sc np})=2\sigma_{\sc
np} \overline{\beta_{\sc np}}+(\rho_{\sc np} -\overline{\rho_{\sc
np}})\kappa_{\sc np} +S(\partial,\xi),  \label{eq-gen-Sach-3}
\end{equation}
\begin{equation}
\xi (\beta_{\sc np}) - \partial (\epsilon_{\sc np})=\sigma_{\sc np}
(\overline{\kappa_{\sc np}}-\overline{\beta_{\sc np}})+\kappa_{\sc np}
(\epsilon_{\sc np} +\overline{\rho_{\sc np}})+\beta_{\sc np} (\epsilon_{\sc np}
-\overline{\rho_{\sc np}})-S(\partial,\xi),  \label{eq-gen-Sach-4}
\end{equation}
\begin{equation}
\partial (\overline{\beta_{\sc np}}) + \overline{\partial}(\beta_{\sc np})=
|\sigma_{\sc np} |^{2}-|\rho_{\sc np}|^{2}-2|\beta_{\sc np} |^{2}-(\rho_{\sc np}
-\overline{\rho_{\sc np}})\epsilon_{\sc np} -S(\partial,\overline{\partial})
+\frac{1}{2}S(\xi,\xi),  \label{eq-gen-Sach-5}
\end{equation}
\begin{equation}
-\xi (\rho_{\sc np}) - \overline{\partial}(\kappa_{\sc np})=|\kappa_{\sc np}
|^{2}+|\sigma_{\sc np} |^{2}+\rho_{\sc np}^{2}+\kappa_{\sc np}
\overline{\beta_{\sc np}}+\frac{1}{2}S(\xi,\xi). \label{eq-gen-Sach-1}
\end{equation}

Let $(M,\varphi,\xi,\eta,g)$ be a $3$-dimensional almost contact metric
manifold, and let $(\partial,\overline{\partial},\xi)$ be a Newman--Penrose
frame. One may choose $\partial$ such that
\begin{equation}
\varphi (\partial) = i\partial,\quad \varphi (\overline{\partial}) =
-i\overline{\partial}. \label{eq-varphi(delta)-varphi(bardelta)}
\end{equation}
Conversely, the almost contact metric structure $(\varphi,\xi,\eta,g)$ is
uniquely determined by the Newman--Penrose frame
$(\partial,\overline{\partial},\xi)$.
An almost contact metric manifold $(M,\varphi,\xi,\eta,g)$ is called a
trans-Sasakian manifold of type $(\alpha,\beta)$ if
\begin{equation}
\kappa_{\sc np} = \sigma_{\sc np} =0, \qquad \rho_{\sc np}=\beta+i\alpha.
\label{eq-tsm-NP}
\end{equation}

\begin{rem-new}\label{rem-tS-cases} (cf.~\cite[Matsuno 2025]{Matsuno-25-arXiv})
A $3$-dimensional trans-Sasakian manifold $(M,\varphi,\xi,\eta,g)$ of type
$(\alpha,\beta)$ reduces to
\begin{itemize}
\item[{\bf (a)}] a ${\cal C}_{6}$-manifold if $\rho_{\sc np}=i\omega_{\sc np}$
(or equivalently $\alpha=\omega_{\sc np}$, $\beta=0$),

\item[{\bf (b)}] an $\alpha$-Sasakian manifold if $\rho_{\sc np}=i\omega_{\sc
np}=$ constant (or equivalently $\alpha=\omega_{\sc np}=$ constant, $\beta=0$),

\item[{\bf (c)}] a Sasakian manifold if $\rho_{\sc np}=i$ (or equivalently
$\alpha=\omega_{\sc np}$, $\beta=0$),

\item[{\bf (d)}] a ${\cal C}_{5}$-manifold if $\rho_{\sc np}=\Theta_{\sc np}$
(or equivalently $\alpha=0$, $\beta=\Theta_{\sc np}$),

\item[{\bf (e)}] a $\beta$-Kenmotsu manifold if $\rho_{\sc np}=\Theta_{\sc
np}=$ constant (or equivalently $\alpha=0$, $\beta=\Theta_{\sc np}=$ constant),

\item[{\bf (f)}] a Kenmotsu manifold $\rho_{\sc np}=1$ (or equivalently
$\alpha=0$, $\beta=\Theta_{\sc np}=1$),

\item[{\bf (g)}] a cosymplectic manifold if $\rho_{\sc np}=0$ (or equivalently
$\alpha=0=\beta$).
\end{itemize}
\end{rem-new}

\section{Trans-Sasakian structures, shear-free geodesic congruences, and
conformal foliations by geodesics\label{Sect-tsm-equiva-NP}}

In dimension $3$, the Newman--Penrose formalism admits a particularly geometric
interpretation of the trans-Sasakian condition. Once the metric and orientation
are fixed, an almost contact metric structure is completely determined by its
structure vector field $\xi$. The covariant derivatives of $\xi$ are then
described by the congruence formed by its integral curves, through the
associated acceleration, shear, expansion, and twist.

The purpose of this subsection is to make this correspondence precise. We
prove, directly from the definitions, that in dimension $3$ the following three
conditions are equivalent:
\begin{enumerate}
\item[{\bf (a)}] a compatible trans-Sasakian almost contact metric structure;

\item[{\bf (b)}] a geodesic shear-free congruence generated by the structure
vector field; and

\item[{\bf (c)}] a conformal foliation by geodesics.
\end{enumerate}
In Newman--Penrose notation, this equivalence is expressed by
\[
\kappa_{\sc np}=0,\qquad \sigma_{\sc np}=0,\qquad \rho_{\sc np}=\beta+i\alpha.
\]
This viewpoint also explains the relation with the work of Baird and Wood (see
\cite{Baird-Wood-91-JAMS}, \cite{Baird-Wood-92-PLMS}), and motivates the
fixed-metric rigidity problem for the homogeneous metrics ${\Bbb
E}(\kappa,\tau)$ treated in the following subsection.

Let $(M^3,g)$ be an oriented $3$-dimensional Riemannian manifold and let $\xi$
be a global unit vector field. Put
\[
\eta=g(\xi,\cdot), \qquad {\cal D}=\xi^\perp=\ker\eta .
\]
The orientation and metric determine the cross product $\times$, and hence an
almost contact metric structure by
\[
\varphi\xi=0,\qquad \varphi X=\xi\times X \quad\text{for }X\in{\cal D}.
\]
Equivalently, if $(e_1,e_2,\xi)$ is a local oriented orthonormal frame, then
\[
\varphi e_1=e_2,\qquad \varphi e_2=-e_1,\qquad \varphi\xi=0.
\]
We use the trans-Sasakian convention recalled in Section~\ref{Sect-Prel},
namely
\begin{equation}\label{eq:TS}
(\nabla_X\varphi)Y = \alpha\{g(X,Y)\xi-\eta(Y)X\} + \beta\{g(\varphi
X,Y)\xi-\eta(Y)\varphi X\}.
\end{equation}

The vector field $\xi$ determines a one-dimensional foliation, locally given by
its flow lines. These flow lines form a geodesic congruence if
\[
\nabla_\xi\xi=0.
\]
The transverse deformation of the congruence is measured by the endomorphism
\[
A:{\cal D}\longrightarrow{\cal D}, \qquad A(X)=\nabla_X\xi .
\]
The image is horizontal because $\xi$ has unit length. With respect to a local
oriented orthonormal frame $(e_1,e_2)$ of ${\cal D}$, write
\begin{equation}\label{eq:matrixA}
\nabla_{e_1}\xi = f_{11} e_1 + f_{12} e_2, \qquad \nabla_{e_2}\xi
= f_{21} e_1 + f_{22} e_2.
\end{equation}
The expansion, twist and shear are the three elementary parts of this
endomorphism. The expansion is
\[
\Theta=\frac12\operatorname{tr}(A)=\frac12(f_{11}+f_{22}),
\]
and the twist is the skew-symmetric part, equivalently
\[
\omega=\frac12(f_{21}-f_{12})
\]
with the above orientation convention.  The shear is the trace-free symmetric
part of $A$. Thus the congruence is shear-free if and only if
\begin{equation}\label{eq:shearfree}
f_{11}=f_{22}, \qquad f_{12}+f_{21}=0.
\end{equation}

\begin{defn}
The foliation generated by $\xi$ is called conformal if the transverse metric
changes conformally along the flow of $\xi$, i.e.
\begin{equation}\label{eq:conformalfoliation}
({\pounds}_\xi g)(X,Y)=2\lambda\, g(X,Y) \qquad \text{for all } X,Y\in{\cal D}
\end{equation}
for some smooth function $\lambda$.  If, in addition, its leaves are geodesics,
we call it a conformal foliation by geodesics.
\end{defn}

\begin{lem}\label{lem:shear-conformal}
The congruence generated by $\xi$ is shear-free if and only if the foliation
generated by $\xi$ is conformal in the sense of \eqref{eq:conformalfoliation}.
In that case $\lambda=\Theta$.
\end{lem}

\noindent{\bf Proof.} For horizontal vector fields $X,Y$, metric compatibility
and the torsion-free property of the Levi-Civita connection give
\[
({\pounds}_\xi g)(X,Y) = g(\nabla_X\xi,Y)+g(X,\nabla_Y\xi).
\]
Hence
\begin{equation}
({\pounds}_\xi g)(X,Y) = 2g(A^{\mathrm{sym}}X,Y), \label{eq-lie-g-Asmy}
\end{equation}
where $A^{\mathrm{sym}}$ is the symmetric part of $A$.

The shear is the trace-free symmetric part of $A$, namely
\[
A^{\mathrm{sym}}-\Theta I, \qquad \Theta=\frac12\operatorname{tr}A.
\]
Therefore the congruence is shear-free if and only if
\begin{equation}
A^{\mathrm{sym}}=\Theta I. \label{eq-Asmy-theta}
\end{equation}
From (\ref{eq-lie-g-Asmy}) and (\ref{eq-Asmy-theta}), we have
\begin{equation}
({\pounds}_\xi g)(X,Y)=2\Theta g(X,Y) \label{eq:conformalfoliation-1}
\end{equation}
for all horizontal vector field $X,Y$, which is precisely the conformality of
the foliation. In view of (\ref{eq:conformalfoliation}) and
(\ref{eq:conformalfoliation-1}), we get $\lambda=\Theta$. $\blacksquare$

We can now state the three equivalent formulations of the same $3$-dimensional
geometry.
\begin{prop}\label{prop:main-equivalence}
Let $(M^3,g)$ be an oriented $3$-dimensional Riemannian manifold and let $\xi$
be a unit vector field.  Let $(\varphi,\xi,\eta,g)$ be the almost contact
metric structure determined by $\xi$ and the orientation as above.  Then the
following are equivalent:
\begin{enumerate}
\item[{\bf (a)}] $(\varphi,\xi,\eta,g)$ is trans-Sasakian of type
$(\alpha,\beta)$ for some smooth functions $\alpha,\beta$; \item[{\bf (b)}] the
congruence generated by $\xi$ is geodesic and shear-free; \item[{\bf (c)}] the
foliation generated by $\xi$ is a conformal foliation by geodesics.
\end{enumerate}
When these conditions hold,
\begin{equation}\label{eq:nablaxiTS}
\nabla_X\xi = -\alpha\varphi X+\beta\bigl(X-\eta(X)\xi\bigr)
\end{equation}
for every vector field $X$, and
\[
\beta=\Theta,\qquad \alpha=\omega .
\]
\end{prop}

\noindent{\bf Proof.} We first prove that (a) implies (b). Suppose that
\eqref{eq:TS} holds. Substituting $Y=\xi$ in (\ref{eq:TS}), we get
\[
(\nabla_X\varphi)\xi =\alpha(\eta(X)\xi-X)-\beta\varphi X,
\]
which gives \eqref{eq:nablaxiTS}. Setting $X=\xi$ in \eqref{eq:nablaxiTS} gives
\[
\nabla_\xi\xi=0,
\]
so the congruence is geodesic. For $X=e_1,e_2$, \eqref{eq:nablaxiTS} gives
\[
\nabla_{e_1}\xi=\beta e_1-\alpha e_2, \qquad \nabla_{e_2}\xi = \alpha e_1+\beta
e_2.
\]
Thus, in the notation of \eqref{eq:matrixA}, one has $f_{11}=f_{22}=\beta$ and
$f_{12}=-\alpha$, $f_{21}=\alpha$.  Hence $f_{11}=f_{22}$ and
$f_{12}+f_{21}=0$, so the shear vanishes.

Conversely, suppose that the congruence is geodesic and shear-free.  Then
$\nabla_\xi\xi=0$ and \eqref{eq:shearfree} holds.  Define
\[
\beta=f_{11}=f_{22}, \qquad \alpha=f_{21}=-f_{12}.
\]
Then
\[
\nabla_{e_1}\xi=\beta e_1-\alpha e_2, \qquad \nabla_{e_2}\xi=\alpha e_1+\beta
e_2, \qquad \nabla_\xi\xi=0,
\]
which is precisely \eqref{eq:nablaxiTS} for all vector fields $X$.

It remains to check that \eqref{eq:nablaxiTS} is equivalent to the full
trans-Sasakian equation \eqref{eq:TS}.  Since $\varphi Y=\xi\times Y$ and the
Levi-Civita connection preserves the cross product,
\[
(\nabla_X\varphi)Y=(\nabla_X\xi)\times Y.
\]
Using \eqref{eq:nablaxiTS} and vector product identity, namely
\[
(\xi\times X)\times Y=X\eta(Y)-\xi g(X,Y),
\]
we get \eqref{eq:TS}. Therefore (a) and (b) are equivalent.  Finally, (b) and
(c) are equivalent by Lemma~\ref{lem:shear-conformal} together with the
geodesic condition. $\blacksquare$

In the Newman--Penrose notation fixed above,
Proposition~\ref{prop:main-equivalence} takes the particularly simple form
\begin{equation}\label{eq:NP-equivalence}
(\varphi,\xi,\eta,g) \text{ is trans-Sasakian of type }(\alpha,\beta)
\quad\Longleftrightarrow\quad \kappa_{\sc np}=0,\quad \sigma_{\sc np}=0, \quad
\rho_{\sc np}=\beta+i\alpha.
\end{equation}
Here $\kappa_{\sc np}=0$ is precisely the geodesic condition, while
$\sigma_{\sc np}=0$ is the shear-free condition. Moreover,
\[
\rho_{\sc np}=\Theta_{\sc np}+i\omega_{\sc np},
\]
so the identities $\beta=\Theta_{\sc np}$ and $\alpha=\omega_{\sc np}$ from
Proposition~\ref{prop:main-equivalence} recover the last condition in
\eqref{eq:NP-equivalence}.

This is the Newman--Penrose characterization stated by Matsuno
\cite[Proposition~3.3]{Matsuno-25-arXiv}. Matsuno also states that a
$3$-dimensional normal almost contact metric structure is characterized by the
structure vector field generating a shear-free geodesic congruence
\cite[Proposition~3.1]{Matsuno-25-arXiv}; in dimension $3$ this is equivalent
to the trans-Sasakian condition.

Proposition~\ref{prop:main-equivalence} moreover reveals an unexpected bridge
between trans-Sasakian geometry and the theory of conformal foliations by
geodesics developed by Baird and Wood in the context of harmonic morphisms. In
dimension $3$, once the metric and orientation are fixed, an oriented unit
vector field $\xi$ determines both an almost contact metric structure and a
one-dimensional foliation, and Proposition~\ref{prop:main-equivalence} shows
that
\[
\text{conformal and geodesic foliation} \quad \Longleftrightarrow\quad
\text{trans-Sasakian structure}.
\]
Thus the classification of conformal foliations by geodesics in $3$-dimensional
space forms in \cite[Baird and Wood 1991]{Baird-Wood-91-JAMS} can be read,
through this dictionary, as a classification of the corresponding
trans-Sasakian structures on those fixed metrics. The isolated singularities
appearing in their work correspond, in our language, to singular congruences,
and therefore lie outside the smooth global trans-Sasakian setting, although
they remain relevant for local examples on punctured domains. There is also a
broader Thurston-geometric context: Baird and Wood relate harmonic morphisms
and conformal geodesic foliations to the natural Seifert fibrations occurring
in Thurston geometry \cite[Baird and Wood 1992]{Baird-Wood-92-PLMS}.

Compared with the foliation-theoretic language of Baird and Wood, the
Newman--Penrose reformulation is particularly powerful: the conformal geodesic
foliation is described through the congruence generated by $\xi$, and its
geometry is encoded entirely by the three scalar coefficients $\kappa_{\sc
np}$, $\sigma_{\sc np}$, and $\rho_{\sc np}$. This is the perspective we use in
Section~6. In the ${\mathbb E}(\kappa,\tau)$-geometries, we fix the homogeneous
metric, but do not assume that the structure vector field is vertical, nor that
the almost contact structure is homogeneous. The Newman--Penrose equations
force verticality in the non-space-form cases. Thus the canonical homogeneous
foliation is not put in by hand; it is recovered from the trans-Sasakian
equations.

\section{Curvature tensors of trans-Sasakian manifolds via the Newman--Penrose
formalism\label{Sect-NP-cur}}

\begin{prop}
Let $\left(M,\varphi,\xi,\eta,g\right) $ be a $3$-dimensional trans-Sasakian
manifold of type $(\alpha,\beta)$ equipped with a Newman--Penrose frame
$(\partial,\overline{\partial},\xi)$. Then,
\begin{equation}
\nabla_{\partial}\xi =\overline{\rho_{\sc np}}\partial, \qquad
\nabla_{\overline{\partial}}\xi =\rho_{\sc np} \overline{\partial},\qquad
\nabla_{\xi}\xi =0, \label{eq-NP-tsm-nabla(xi)}
\end{equation}
\begin{equation}
\left(\nabla_{\partial}\eta \right) \partial = 0, \qquad
\left(\nabla_{\partial}\eta \right) \overline{\partial}=\overline{\rho_{\sc
np}},\qquad \left(\nabla_{\partial}\eta \right) \xi
=0,\label{eq-NP-tsm-nabla(delta)eta}
\end{equation}
\begin{equation}
\left(\nabla_{\xi}\eta \right) \partial = 0, \qquad \left(\nabla_{\xi}\eta
\right) {\overline{\partial}} = 0,\qquad \left(\nabla_{\xi}\eta\right) \xi =0.
\label{eq-NP-tsm-nabla(xi)eta}
\end{equation}
\end{prop}

\noindent {\bf Proof}. From (\ref{eq-nabla_X(xi)}), (\ref{eq-nabla_X(eta)}),
(\ref{eq-NPF-metric}), (\ref{eq-eta-xi-partial}),
(\ref{eq-varphi(delta)-varphi(bardelta)}) and (\ref{eq-tsm-NP}), we directly
obtain (\ref{eq-NP-tsm-nabla(xi)}), (\ref{eq-NP-tsm-nabla(delta)eta}) and
(\ref{eq-NP-tsm-nabla(xi)eta}). $\blacksquare$

\begin{ex-new}
Let
\[
M=\{\left(x^{1},x^{2},x^{3}\right) \in {\Bbb R}^{3}:x^{3}> 0\}
\]
be a $3$-dimensional smooth manifold, where $\left( x^{1},x^{2},x^{3}\right)$
are the standard coordinates in ${\Bbb R}^{3}$. Define a linearly independent
global coordinate frame $(B_{1},B_{2},B_{3})$ on $M$ by
\[
B_{1} = e^{x^{3}}\left(\frac{\partial}{\partial
x^{1}}+x^{2}\frac{\partial}{\partial x^{3}}\right),\quad
B_{2}=e^{x^{3}}\frac{\partial}{\partial x^{2}},\quad
B_{3}=\frac{\partial}{\partial x^{3}}.
\]
Now, define the associated Newman--Penrose frame
$(\partial,\overline{\partial},\xi)$ by
\[
\partial =\frac{1}{\sqrt{2}}\left(B_{1}-iB_{2}\right),
\quad \overline{\partial}=\frac{1}{\sqrt{2}}\left(B_{1} +iB_{2}\right), \quad
\xi =B_{3}.
\]
Define the $1$-form $\eta$ on $M$ by
\[
\eta(X)=g(X,B_{3})
\]
for every vector field $X$ on $M$. Let $\varphi$ be the $(1,1)$ tensor field
given by
\[
\varphi (\partial)=i\partial,\quad \varphi (\overline{\partial}) =
-i\overline{\partial},\quad \varphi \xi=0.
\]
Consider the Riemannian metric $g$ on $M$ defined by
\[
g(\xi,\xi)=g(\partial,\overline{\partial})=1, \qquad
g(\xi,\partial)=g(\partial,\partial)=
g(\overline{\partial},\overline{\partial})=0.
\]
Let $\nabla$ denote the Levi-Civita connection of $g$. Then, by a direct
calculation, we obtain
\[
[B_{1},B_{2}] = x^{2}e^{x^{3}}B_{2}-e^{2x^{3}}B_{3}, \quad
[B_{1},B_{3}]=-B_{1}, \quad [B_{2},B_{3}]=-B_{2},
\]
\[
\nabla_{B_{1}}\xi = -B_{1}+\frac{1}{2}e^{2x^{3}}B_{2},\quad \nabla_{B_{2}}\xi =
-B_{2}-\frac{1}{2}e^{2x^{3}}B_{1},\quad \nabla_{\xi}\xi =0,
\]
\[
\nabla_{\partial}\xi =  \left(-1+\frac{i}{2}e^{2x^{3}}\right)\partial,\quad
\nabla_{\overline{\partial}}\xi =
\left(-1-\frac{i}{2}e^{2x^{3}}\right)\overline{\partial}.
\]
Using the above computation in (\ref{eq-spin-coe-1}), we obtain
\[
\sigma_{\sc np} = -g\left(\partial,\nabla_{\partial}\xi \right) =0, \quad
\kappa_{\sc np} = -g\left(\partial,\nabla_{\xi}\xi \right) =0, \quad \rho_{\sc
np} = g\left(\partial,\nabla_{\overline{\partial}}\xi \right)
=-1-\frac{i}{2}e^{2x^{3}}.
\]
Thus, $(\varphi,\xi,\eta,g)$ defines a trans-Sasakian structure on $M$ of type
\[
(\alpha,\beta)=\left(-\frac{1}{2}e^{2x^{3}},-1\right).
\]
Consequently, integral curves of $\xi$ generate a congruence which is geodesic
and shear-free. Moreover, the expansion $\Theta_{\sc np}$ and twist
$\omega_{\sc np}$ are given by
\[
\Theta_{\sc np} =-1,\quad \omega_{\sc np} =-\frac{1}{2}e^{2x^{3}}.
\]
\end{ex-new}

\begin{ex-new}
Let
\[
M=\{\left(x^{1},x^{2},x^{3}\right) \in {\Bbb R}^{3}:x^{3}> 0\}
\]
be a $3$-dimensional smooth manifold, where $\left(x^{1},x^{2},x^{3}\right)$
are the standard coordinates in ${\Bbb R}^{3}$. Define a linearly independent
global coordinate frame $(B_{1},B_{2},B_{3})$ on $M$ by
\[
B_{1} = x^{3}\left(\frac{\partial}{\partial x^{1}} +
x^{2}\frac{\partial}{\partial x^{3}}\right), \quad
B_{2}=x^{3}\frac{\partial}{\partial x^{2}},\quad B_{3} =
\frac{\partial}{\partial x^{3}}.
\]
Now, define the associated Newman--Penrose frame
$(\partial,\overline{\partial},\xi)$ by
\[
\partial = \frac{1}{\sqrt{2}}(B_{1}-iB_{2}), \quad
\overline{\partial}=\frac{1}{\sqrt{2}}(B_{1}+iB_{2}), \quad \xi =B_{3}.
\]
Define the $1$-form $\eta $ on $M$ by
\[
\eta (X)=g(X,B_{3})
\]
for every vector field $X$ on $M$. Let $\varphi $ be the $(1,1)$ tensor field
given by
\[
\varphi (\partial)=i\partial,\quad \varphi
(\overline{\partial})=-i\overline{\partial },\quad \varphi (\xi )=0.
\]
Consider the Riemannian metric $g$ on $M$ defined by
\begin{equation}
g(\xi,\xi) = g(\partial,\overline{\partial})=1,\qquad
g(\xi,\partial)=g(\partial,\partial) =
g(\overline{\partial},\overline{\partial}) = 0. \nonumber
\end{equation}
Let $\nabla$ denote the Levi-Civita connection of $g$. Then, by a direct
calculation, we obtain
\[
[B_{1},B_{2}] = x^{2}B_{2}-{(x^{3})}^{2}B_{3}, \quad [B_{1},B_{3}] =
-\frac{1}{x^{3}}B_{1}, \quad [B_{2},B_{3}]=-\frac{1}{x^{3}}B_{2},
\]
\[
\nabla_{B_{1}}\xi =-\frac{1}{x^{3}}B_{1}+\frac{1}{2}(x^{3})^{2}B_{2},\quad
\nabla_{B_{2}}\xi =-\frac{1}{x^{3}}B_{2}-\frac{1}{2}(x^{3})^{2}B_{1},\quad
\nabla_{\xi}\xi =0,
\]
\[
\nabla_{\partial}\xi =
\left(-\frac{1}{x^{3}}+i\frac{(x^{3})^{2}}{2}\right)\partial, \quad
\nabla_{\overline{\partial}}\xi = \left(-\frac{1}{x^{3}}-
i\frac{(x^{3})^{2}}{2}\right)\overline{\partial}.
\]
Using the above computation in (\ref{eq-spin-coe-1}), we obtain
\[
\sigma_{\sc np} =-g(\partial,\nabla_{\partial}\xi)=0,
\]
\[
\kappa_{\sc np} =-g(\partial,\nabla_{\xi}\xi)=0,
\]
\[
\rho_{\sc np} =g(\partial,\nabla_{\overline{\partial}}\xi)=-\frac{1}{x^{3}}
-\frac{i}{2}(x^{3})^{2}.
\]
Thus, $(\varphi,\xi,\eta,g)$ defines a trans-Sasakian structure on $M$ of type
\[
(\alpha,\beta)=\left(-\tfrac{1}{2}(x^{3})^{2},-\tfrac{1}{x^{3}}\right).
\]
Consequently, integral curves of $\xi$ generate a congruence which is geodesic
and shear-free. Moreover, the expansion $\Theta_{\sc np}$ and twist
$\omega_{\sc np}$ are given by
\[
\Theta_{\sc np} =-\frac{1}{x^{3}},\quad \omega_{\sc np} =
-\frac{1}{2}(x^{3})^{2}.
\]
\end{ex-new}

\begin{prop}
Let $\left(M,\varphi,\xi,\eta,g\right) $ be a $3$-dimensional trans-Sasakian
manifold of type $(\alpha,\beta)$ equipped with a Newman--Penrose frame
$(\partial,\overline{\partial},\xi)$. Then,
\begin{equation}
{\rm d}\alpha\left(\partial \right) =\frac{1}{2i}\left(\partial\left(\rho_{\sc
np} \right) -\partial \left(\overline{\rho_{\sc np}}\right)\right)
=\frac{1}{2i}\left(\eth (\rho_{\sc np})-\eth (\overline{\rho_{\sc
np}})\right),\label{eq-NP-dalphas(delta)}
\end{equation}
\begin{equation}
{\rm d}\alpha\left(\overline{\partial}\right)
=\frac{1}{2i}\left(\overline{\partial}\left(\rho_{\sc np} \right)
-\overline{\partial}\left(\overline{\rho_{\sc np}}\right)\right)
=\frac{1}{2i}\left(\overline{\eth}(\rho_{\sc
np})-\overline{\eth}(\overline{\rho_{\sc np}})\right),
\label{eq-NP-dalphas(bardelta)}
\end{equation}
\begin{equation}
{\rm d}\alpha\left(\xi \right) =\frac{1}{2i}\left(\xi \left(\rho_{\sc np}
\right) -\xi \left(\overline{\rho_{\sc np}}\right) \right)
=\frac{1}{2i}\left(P(\rho_{\sc np})-P(\overline{\rho_{\sc np}})\right),
\label{eq-NP-dalphas(xi)}
\end{equation}
\begin{equation}
{\rm d}\beta\left(\partial \right) =\frac{1}{2}\left(\partial\left(\rho_{\sc
np} \right) +\partial \left(\overline{\rho_{\sc np}}\right) \right)
=\frac{1}{2}\left(\eth (\rho_{\sc np})+\eth (\overline{\rho_{\sc
np}})\right),\label{eq-NP-dbetas(delta)}
\end{equation}
\begin{equation}
{\rm d}\beta\left(\overline{\partial}\right)
=\frac{1}{2}\left(\overline{\partial}\left(\rho_{\sc np} \right)
+\overline{\partial}\left(\overline{\rho_{\sc np}}\right) \right)
=\frac{1}{2}\left(\overline{\eth}(\rho_{\sc
np})+\overline{\eth}(\overline{\rho_{\sc np}})\right),
\label{eq-NP-dbetas(deltabar)}
\end{equation}
\begin{equation}
{\rm d}\beta\left(\xi \right) =\frac{1}{2}\left(\xi\left(\rho_{\sc np}\right)
+\xi \left(\overline{\rho_{\sc np}}\right) \right)
=\frac{1}{2}\left(P(\rho_{\sc np})+P(\overline{\rho_{\sc np}})\right) .
\label{eq-NP-dbetas(xi)}
\end{equation}
\end{prop}

\noindent {\bf Proof.} It is known that the spin coefficient $\rho_{\sc
np}=\beta+i\alpha$ on $M$. Consequently, its real and imaginary parts are given
by
\begin{equation}
\alpha=\frac{1}{2i}\left(\rho_{\sc np} -\overline{\rho_{\sc np}}\right),\qquad
\beta=\frac{1}{2}\left(\rho_{\sc np} +\overline{\rho_{\sc np}}\right).
\label{eq-alphas-betas-NP}
\end{equation}
Using (\ref{eq-spin-weight2}), (\ref{eq-spin-weight3}), (\ref{eq-spin-weight1})
together with (\ref{eq-alphas-betas-NP}), we obtain
(\ref{eq-NP-dalphas(delta)}), (\ref{eq-NP-dalphas(bardelta)}),
(\ref{eq-NP-dalphas(xi)}), (\ref{eq-NP-dbetas(delta)}),
(\ref{eq-NP-dbetas(deltabar)}) and (\ref{eq-NP-dbetas(xi)}). $\blacksquare$

\begin{prop}
Let $\left(M,\varphi,\xi,\eta,g\right) $ be a $3$-dimensional trans-Sasakian
manifold of type $(\alpha,\beta)$ equipped with a Newman--Penrose frame
$(\partial,\overline{\partial},\xi)$. Then,
\begin{equation}
\rho_{\sc np}^{2}-\left(\overline{\rho_{\sc np}}\right)^{2}+\xi \left(\rho_{\sc
np} \right) -\xi \left(\overline{\rho_{\sc np}}\right) =0.
\label{eq-NP-2alphabeta+xi(alpha)=0}
\end{equation}
\end{prop}

\noindent {\bf Proof.} By using (\ref{eq-NP-dalphas(xi)}) and
(\ref{eq-alphas-betas-NP}) in (\ref{eq-condition-alpha-beta-xi}), we get
(\ref{eq-NP-2alphabeta+xi(alpha)=0}). $\blacksquare$

\begin{prop}
Let $\left(M,\varphi,\xi,\eta,g\right) $ be a $3$-dimensional trans-Sasakian
manifold of type $(\alpha,\beta)$ equipped with a Newman--Penrose frame
$(\partial,\overline{\partial},\xi)$. Then,
\begin{equation}
R\left(\partial,\overline{\partial}\right) \xi =
-\overline{\partial}(\overline{\rho_{\sc np}})\partial +\partial (\rho_{\sc
np})\overline{\partial}, \label{eq-NP-tsm-R(delta,bardelta)}
\end{equation}
\begin{equation}
R\left(\partial, \xi\right) \xi = -\left(\overline{\rho_{\sc
np}}\right)^{2}\partial-\xi (\overline{\rho_{\sc np}})\partial,
\label{eq-NP-tsm-R(xi,delta)}
\end{equation}
\begin{equation}
R\left(\overline{\partial},\xi\right) \xi = -\rho_{\sc
np}^{2}\overline{\partial}-\xi (\rho_{\sc np})\overline{\partial}.
\label{eq-NP-tsm-R(xi,bardelta)}
\end{equation}
\end{prop}

\noindent {\bf Proof.} From (\ref{eq-tr-sasa-cur-ope}),
(\ref{eq-eta-xi-partial}) and (\ref{eq-varphi(delta)-varphi(bardelta)}), we
obtain
\begin{equation}
R\left(\partial,\overline{\partial}\right) \xi = i{\rm
d}\alpha\left(\overline{\partial}\right) \partial +i{\rm
d}\alpha\left(\partial\right) \overline{\partial}-{\rm
d}\beta\left(\overline{\partial}\right)\partial +{\rm
d}\beta\left(\partial\right) \overline{\partial},
\label{eq-NP-tsm-R(delta,bardelta)-1}
\end{equation}
\begin{equation}
R\left(\partial, \xi\right) \xi = \left({\alpha}^{2} -{\beta}^{2}\right)
\partial +2i\alpha\beta\partial +i{\rm d}\alpha\left(\xi \right) \partial -{\rm
d}\beta\left(\xi\right) \partial,  \label{eq-NP-tsm-R(xi,delta)-1}
\end{equation}
\begin{equation}
R\left(\overline{\partial},\xi\right) \xi =  \left({\alpha}^{2}-
{\beta}^{2}\right) \overline{\partial}-2i\alpha\beta\overline{\partial}-i{\rm
d}\alpha\left(\xi \right) \overline{\partial}-{\rm d}\beta\left(\xi \right)
\overline{\partial}.  \label{eq-NP-tsm-R(xi,bardelta)-1}
\end{equation}
Using (\ref{eq-NP-dalphas(delta)}), (\ref{eq-NP-dalphas(bardelta)}),
(\ref{eq-NP-dbetas(delta)}) and (\ref{eq-NP-dbetas(deltabar)}) in
(\ref{eq-NP-tsm-R(delta,bardelta)-1}), we get
(\ref{eq-NP-tsm-R(delta,bardelta)}). Further, using (\ref{eq-NP-dalphas(xi)}),
(\ref{eq-NP-dbetas(xi)}) and (\ref{eq-alphas-betas-NP}) in
(\ref{eq-NP-tsm-R(xi,delta)-1}) and (\ref{eq-NP-tsm-R(xi,bardelta)-1}), we
obtain (\ref{eq-NP-tsm-R(xi,delta)}) and (\ref{eq-NP-tsm-R(xi,bardelta)}),
respectively. $\blacksquare$

\begin{prop}
Let $\left(M,\varphi,\xi,\eta,g\right) $ be a $3$-dimensional trans-Sasakian
manifold of type $(\alpha,\beta)$ equipped with a Newman--Penrose frame
$(\partial,\overline{\partial},\xi)$. Then the Ricci tensor $S$ of $M$ is given
by
\begin{equation}
S(\partial,\partial)=0,  \label{eq-tsm-gen-Sach-2}
\end{equation}
\begin{equation}
S(\partial,\xi)=-\partial (\rho_{\sc np})=\beta_{\sc np}  (\epsilon_{\sc np}
-\overline{\rho_{\sc np}})-\xi (\beta_{\sc np})+\partial(\epsilon_{\sc np}),
\label{eq-tsm-gen-Sach-3}
\end{equation}
\begin{equation}
S\left(\xi,\xi \right) =-2\xi \left(\rho_{\sc np} \right) -2\rho_{\sc np}^{2},
\label{eq-tsm-gen-Sach-1}
\end{equation}
\begin{equation}
S(\partial,\overline{\partial}) = -|\rho_{\sc np} |^{2}-2|\beta_{\sc np}
|^{2}-(\rho_{\sc np} -\overline{\rho_{\sc np}})\epsilon_{\sc np} -\xi
\left(\rho_{\sc np} \right) -\rho_{\sc np}^{2}-\partial (\overline{\beta_{\sc
np}})-\overline{\partial}(\beta_{\sc np}).  \label{eq-tsm-gen-Sach-5}
\end{equation}
\end{prop}

\noindent {\bf Proof.} Substituting (\ref{eq-tsm-NP}) into
(\ref{eq-gen-Sach-2}), we obtain (\ref{eq-tsm-gen-Sach-2}). Similarly, using
(\ref{eq-tsm-NP}) in (\ref{eq-gen-Sach-3}) and (\ref{eq-gen-Sach-4}), we derive
(\ref{eq-tsm-gen-Sach-3}). Further, substituting (\ref{eq-tsm-NP}) into
(\ref{eq-gen-Sach-1}) yields (\ref{eq-tsm-gen-Sach-1}). Finally, using
(\ref{eq-tsm-NP}) and (\ref{eq-tsm-gen-Sach-1}) in (\ref{eq-gen-Sach-5}), we
obtain (\ref{eq-tsm-gen-Sach-5}). $\blacksquare$

\begin{cor}
Let $\left(M,\varphi,\xi,\eta,g\right)$ be a $3$-dimensional Sasakian manifold
equipped with a Newman--Penrose frame $(\partial,\overline{\partial},\xi)$.
Then the Ricci tensor $S$ of $M$ is given by
\begin{equation}
S(\partial,\partial)=0,  \label{eq-sm-gen-Sach-2}
\end{equation}
\begin{equation}
S(\partial,\xi)=\beta_{\sc np} (\epsilon_{\sc np} +i) - \xi (\beta_{\sc
np})+\partial(\epsilon_{\sc np})=0, \label{eq-sm-gen-Sach-3}
\end{equation}
\begin{equation}
S\left(\xi,\xi \right) =2,  \label{eq-sm-gen-Sach-1}
\end{equation}
\begin{equation}
S(\partial,\overline{\partial})=-2|\beta_{\sc np} |^{2}-2i\epsilon_{\sc np}
-\partial (\overline{\beta_{\sc np}})-\overline{\partial}(\beta_{\sc np}).
\label{eq-sm-gen-Sach-5}
\end{equation}
\end{cor}

\begin{cor}
Let $\left(M,\varphi,\xi,\eta,g\right)$ be a $3$-dimensional Kenmotsu manifold
equipped with a Newman--Penrose frame $(\partial,\overline{\partial},\xi)$.
Then the Ricci tensor $S$ of $M$ is given by
\begin{equation}
S(\partial,\partial)=0,  \label{eq-Kenm-gen-Sach-2}
\end{equation}
\begin{equation}
S(\partial,\xi)=\beta_{\sc np} (\epsilon_{\sc np} -1) - \xi (\beta_{\sc
np})+\partial(\epsilon_{\sc np})=0, \label{eq-Kenm-gen-Sach-3}
\end{equation}
\begin{equation}
S\left(\xi,\xi \right) =-2,  \label{eq-Kenm-gen-Sach-1}
\end{equation}
\begin{equation}
S(\partial,\overline{\partial})=-2-2|\beta_{\sc np}
|^{2}-\partial(\overline{\beta_{\sc np}})-\overline{\partial}(\beta_{\sc np}).
\label{eq-Kenm-gen-Sach-5}
\end{equation}
\end{cor}

\begin{cor}
Let $\left(M,\varphi,\xi,\eta,g\right)$ be a $3$-dimensional cosymplectic
manifold equipped with a Newman--Penrose frame
$(\partial,\overline{\partial},\xi)$. Then the Ricci tensor $S$ of $M$ is given
by
\begin{equation}
S(\partial,\partial)=0,  \label{eq-cosym-gen-Sach-2}
\end{equation}
\begin{equation}
S(\partial,\xi)=\beta_{\sc np} (\epsilon_{\sc np}) - \xi (\beta_{\sc
np})+\partial (\epsilon_{\sc np})=0, \label{eq-cosym-gen-Sach-3}
\end{equation}
\begin{equation}
S\left(\xi,\xi \right) =0,  \label{eq-cosym-gen-Sach-1}
\end{equation}
\begin{equation}
S(\partial,\overline{\partial})=-2|\beta_{\sc np} |^{2} - \partial
(\overline{\beta_{\sc np}})-\overline{\partial}(\beta_{\sc np}).
\label{eq-cosym-gen-Sach-5}
\end{equation}
\end{cor}

We recall that, on a $3$-dimensional Riemannian manifold, the second Bianchi
identity yields the following equations \cite[Aazami 2015]{Aazami-15-JGP}:
\begin{eqnarray}
&&\partial (S(\overline{\partial},\xi)) + \overline{\partial}(S(\partial,\xi))
- \xi (S(\partial,\overline{\partial})) - \frac{1}{2}\xi(S(\xi,\xi))  \nonumber \\
&&\hspace{1cm}=-(\rho_{\sc np} +\overline{\rho_{\sc np}})
(S(\xi,\xi)-S(\partial,\partial))-\overline{\sigma_{\sc np}}
S(\partial,\partial)  \nonumber \\
&&\hspace{1.2cm}-\sigma_{\sc np} S(\partial,\overline{\partial})
-(2\overline{\kappa_{\sc np}}+\overline{\beta_{\sc np}})S(\partial,\xi)
-(2\kappa_{\sc np}+\beta_{\sc np})S(\partial,\overline{\partial}),
\label{eq-diff-bianchi-2}
\end{eqnarray}
\begin{eqnarray}
&&\xi (S(\partial,\xi))-\frac{1}{2}\partial (S(\xi,\xi)) +
\overline{\partial}(S(\partial,\partial))  \nonumber \\
&&\hspace{1cm}=\kappa_{\sc np} S(\xi,\xi)+(\epsilon_{\sc np}
-2\overline{\rho_{\sc np}}-\rho_{\sc np})S(\partial,\xi)  \nonumber \\
&&\hspace{1.2cm}+\sigma_{\sc np} S(\partial,\overline{\partial})
-(\overline{\kappa_{\sc np}} +2\overline{\beta_{\sc np}})S(\partial,\partial)
-\kappa_{\sc np} S(\partial,\overline{\partial}).\label{eq-diff-bianchi-1}
\end{eqnarray}

\begin{prop}
Let $\left(M,\varphi,\xi,\eta,g\right)$ be a $3$-dimensional trans-Sasakian
manifold of type $(\alpha,\beta)$ equipped with a Newman--Penrose frame
$(\partial,\overline{\partial},\xi)$. Then the second Bianchi identity can be
expressed in the following form:
\begin{eqnarray}
&&\partial (S(\overline{\partial},\xi)) +\overline{\partial}(S(\xi,\partial))
-\xi (S(\partial,\overline{\partial})) -\frac{1}{2}\xi(S(\xi,\xi))  \nonumber\\
&&\qquad =-(\rho_{\sc np} +\overline{\rho_{\sc np}})(S(\xi,\xi))
-\overline{\beta_{\sc np}}S(\partial,\xi)-\beta_{\sc np} S(\partial,
\overline{\partial}),\label{eq-PN-tsm-diff-bianchi-2}
\end{eqnarray}
\begin{equation}
\xi (S(\partial,\xi)) - \frac{1}{2}\partial (S(\xi,\xi)) = (\epsilon_{\sc np}
-2\overline{\rho_{\sc np}}-\rho_{\sc np})S(\partial,\xi).
\label{eq-PN-tsm-diff-bianchi-1}
\end{equation}
\end{prop}

\noindent {\bf Proof.} Using (\ref{eq-tsm-NP}) in (\ref{eq-diff-bianchi-2}) and
(\ref{eq-diff-bianchi-1}), we obtain (\ref{eq-PN-tsm-diff-bianchi-2}) and
(\ref{eq-PN-tsm-diff-bianchi-1}), respectively. $\blacksquare$

Let $(M,g)$ be a Riemannian manifold. Then the scalar curvature $\tau$ of $M$
is given by
\begin{equation}
\tau=\frac{1}{2}{\rm tr}(S),  \label{eq-tau}
\end{equation}
where $S$ is the Ricci tensor of $M$.

\begin{lem}
Let $(M,g)$ be a $3$-dimensional Riemannian manifold equipped with a
Newman--Penrose frame $(\partial,\overline{\partial},\xi)$. Then the scalar
curvature $\tau$ of $M$ is given by
\begin{equation}
\tau = S(\partial,\overline{\partial}) + \frac{1}{2} S(\xi,\xi),
\label{eq-scalar-NP}
\end{equation}
where $S$ is the Ricci tensor of $M$.
\end{lem}

\noindent {\bf Proof.} Recall that $(e_1,e_2,\xi)$ be a real orthonormal frame
on $M$. Then from (\ref{eq-tau}), we have
\begin{equation}
\tau = \frac{1}{2}\{S(e_1,e_1)+S(e_2,e_2)+S(\xi,\xi)\}.
\label{eq-scalar-NP-1}
\end{equation}
From (\ref{eq-partial-partial-bar}), we have
\begin{equation}
e_{1}=\frac{1}{\sqrt{2}}(\partial+\overline{\partial}), \qquad
e_{2}=\frac{i}{\sqrt{2}}(\partial-\overline{\partial}).
\label{eq-scalar-NP-2}
\end{equation}
Using (\ref{eq-scalar-NP-2}) in (\ref{eq-scalar-NP-1}), we get
(\ref{eq-scalar-NP}). $\blacksquare$

\begin{prop}
Let $\left(M,\varphi,\xi,\eta,g\right) $ be a $3$-dimensional trans-Sasakian
manifold of type $(\alpha,\beta)$ equipped with a Newman--Penrose frame
$(\partial,\overline{\partial},\xi)$. Then the scalar curvature $\tau$ of $M$
is given by
\begin{equation}
\tau =-2\xi \left(\rho_{\sc np} \right) -2\rho_{\sc np}^{2} -|\rho_{\sc np}
|^{2}-2|\beta_{\sc np} |^{2} -(\rho_{\sc np} -\overline{\rho_{\sc
np}})\epsilon_{\sc np} - \partial (\overline{\beta_{\sc np}})-
\overline{\partial}(\beta_{\sc np}).  \label{eq-NP-tsm-scalar-cur}
\end{equation}
\end{prop}

\noindent {\bf Proof.} Substituting (\ref{eq-tsm-gen-Sach-1}) and
(\ref{eq-tsm-gen-Sach-5}) into (\ref{eq-scalar-NP}), we get
(\ref{eq-NP-tsm-scalar-cur}). $\blacksquare$

\begin{cor}
Let $\left(M,\varphi,\xi,\eta,g\right) $ be a $3$-dimensional Sasakian manifold
equipped with a Newman--Penrose frame $(\partial,\overline{\partial},\xi)$.
Then the scalar curvature $\tau$ of $M$ is given by
\begin{equation}
\tau = 1-2|\beta_{\sc np} |^{2}-2i\epsilon_{\sc np} -\partial
(\overline{\beta_{\sc np}})-\overline{\partial}(\beta_{\sc np}).
\label{eq-NP-sm-scalar-cur}
\end{equation}
\end{cor}

\begin{cor}
Let $\left(M,\varphi,\xi,\eta,g\right) $ be a $3$-dimensional Kenmotsu manifold
equipped with a Newman--Penrose frame $(\partial,\overline{\partial},\xi)$.
Then the scalar curvature $\tau$ of $M$ is given by
\begin{equation}
\tau =-3-2|\beta_{\sc np} |^{2}-\partial (\overline{\beta_{\sc np}})-
\overline{\partial}(\beta_{\sc np}).  \label{eq-NP-Ken-scalar-cur}
\end{equation}
\end{cor}

\begin{cor}
Let $\left(M,\varphi,\xi,\eta,g\right) $ be a $3$-dimensional cosymplectic
manifold equipped with a Newman--Penrose frame
$(\partial,\overline{\partial},\xi)$. Then the scalar curvature $\tau$ of $M$
is given by
\begin{equation}
\tau =-2|\beta_{\sc np} |^{2}-\partial (\overline{\beta_{\sc np}})
-\overline{\partial}(\beta_{\sc np}). \label{eq-NP-cosym-scalar-cur}
\end{equation}
\end{cor}

The Kulkarni-Nomizu product $T_{1}\circledast T_{2}$ of any $(0,2)$-tensors
$T_{1}$ and $T_{2}$ on a smooth manifold $M$ is a $(0,4)$-tensor defined by
(see \cite[Tripathi 2025, Eq.~(3.1)]{Tri-25-CM})
\begin{eqnarray*}
T_{1}\circledast T_{2}(X,Y,Z,W) &=&T_{1}(Y,Z)T_{2}(X,W)-T_{1}(X,Z)T_{2}(Y,W)
\nonumber\\ &&+T_{2}(Y,Z)T_{1}(X,W)-T_{2}(X,Z)T_{1}(Y,W)
\end{eqnarray*}
for all vector fields $X$, $Y$, $Z$, $W$ on $M$. Now, on any $3$-dimensional
Riemannian manifold $(M,g)$, the Riemann-Christoffel curvature tensor $R$ can
be given by (cf. \cite[Blair {\em et al.} 1990]{Blair-Kouf-Sharma-90-Kodai},
\cite[Tripathi 2026, Theorem~1.37]{Tri-26-Book})
\begin{equation}
R = g \circledast \left(S-\frac{\tau}{2}g\right),  \label{eq-3D-Riem-cur}
\end{equation}
where $S$ is the Ricci tensor of $M$, and $\tau$ is the scalar curvature of $M$
defined by (\ref{eq-tau}).

\begin{prop}
Let $\left(M,\varphi,\xi,\eta,g\right) $ be a $3$-dimensional trans-Sasakian
manifold of type $(\alpha,\beta)$ equipped with a Newman--Penrose frame
$(\partial,\overline{\partial},\xi)$. Then the Riemann--Christoffel curvature
tensor $R$ of $M$ is given by
\begin{equation}
R(\partial,\overline{\partial},\overline{\partial},\partial)  = |\rho_{\sc
np}|^{2}+2|\beta_{\sc np}|^{2}+(\rho_{\sc np} -\overline{\rho_{\sc
np}})\epsilon_{\sc np} +\partial (\overline{\beta_{\sc
np}})+\overline{\partial}(\beta_{\sc np}),  \label{eq-NP-tsm-Riem-cur-1}
\end{equation}
\begin{equation}
R(\partial,\overline{\partial},\xi,\partial) = \partial (\rho_{\sc np}),
\label{eq-NP-tsm-Riem-cur-4}
\end{equation}
\begin{equation}
R(\partial,\xi,\xi,\partial) = 0,  \label{eq-NP-tsm-Riem-cur-2}
\end{equation}
\begin{equation}
R(\partial,\xi,\xi,\overline{\partial}) = -\xi \left(\rho_{\sc np} \right)
-\rho_{\sc np}^{2}. \label{eq-NP-tsm-Riem-cur-3}
\end{equation}
\end{prop}

\noindent {\bf Proof.} By a direct computation using (\ref{eq-NPF-metric}),
(\ref{eq-tsm-gen-Sach-3}), (\ref{eq-tsm-gen-Sach-1}), (\ref{eq-tsm-gen-Sach-5})
and (\ref{eq-NP-tsm-scalar-cur}) in (\ref{eq-3D-Riem-cur}), we obtain
(\ref{eq-NP-tsm-Riem-cur-1}), (\ref{eq-NP-tsm-Riem-cur-4}),
(\ref{eq-NP-tsm-Riem-cur-2}), and (\ref{eq-NP-tsm-Riem-cur-3}). $\blacksquare$

\begin{cor}
Let $\left(M,\varphi,\xi,\eta,g\right)$ be a $3$-dimensional Sasakian manifold
equipped with a Newman--Penrose frame $(\partial,\overline{\partial},\xi)$.
Then the Riemann--Christoffel curvature tensor $R$ of $M$ is given by
\begin{equation}
R(\partial,\overline{\partial},\overline{\partial},\partial) = 1+2|\beta_{\sc
np} |^{2}+2i\epsilon_{\sc np} +\partial (\overline{\beta_{\sc
np}})+\overline{\partial}(\beta_{\sc np}), \label{eq-NP-sm-Riem-cur-1}
\end{equation}
\begin{equation}
R(\partial,\overline{\partial},\xi,\partial) = 0, \label{eq-NP-sm-Riem-cur-4}
\end{equation}
\begin{equation}
R(\partial,\xi,\xi,\partial) = 0,  \label{eq-NP-sm-Riem-cur-2}
\end{equation}
\begin{equation}
R(\partial,\xi,\xi,\overline{\partial}) = 1.  \label{eq-NP-sm-Riem-cur-3}
\end{equation}
\end{cor}

\begin{cor}
Let $\left(M,\varphi,\xi,\eta,g\right) $ be a $3$-dimensional Kenmotsu manifold
equipped with a Newman--Penrose frame $(\partial,\overline{\partial},\xi)$.
Then the Riemann--Christoffel curvature tensor $R$ of $M$ is given by
\begin{equation}
R(\partial,\overline{\partial},\overline{\partial},\partial) = 1+2|\beta_{\sc
np} |^{2}+\partial (\overline{\beta_{\sc np}})+\overline{\partial}(\beta_{\sc
np}), \label{eq-NP-Ken-Riem-cur-1}
\end{equation}
\begin{equation}
R(\partial,\overline{\partial},\xi,\partial) = 0, \label{eq-NP-Ken-Riem-cur-4}
\end{equation}
\begin{equation}
R(\partial,\xi,\xi,\partial) = 0,  \label{eq-NP-Ken-Riem-cur-2}
\end{equation}
\begin{equation}
R(\partial,\xi,\xi,\overline{\partial}) = -1.  \label{eq-NP-Ken-Riem-cur-3}
\end{equation}
\end{cor}

\begin{cor}
Let $\left(M,\varphi,\xi,\eta,g\right) $ be a $3$-dimensional cosymplectic
manifold equipped with a Newman--Penrose frame
$(\partial,\overline{\partial},\xi)$. Then the Riemann--Christoffel curvature
tensor $R$ of $M$ is given by
\begin{equation}
R(\partial,\overline{\partial},\overline{\partial},\partial) = 2|\beta_{\sc
np}|^{2}+\partial (\overline{\beta_{\sc np}})+\overline{\partial}(\beta_{\sc
np}), \label{eq-NP-cosym-Riem-cur-1}
\end{equation}
\begin{equation}
R(\partial,\overline{\partial},\xi,\partial) = 0,
\label{eq-NP-cosym-Riem-cur-4}
\end{equation}
\begin{equation}
R(\partial,\xi,\xi,\partial) = 0,  \label{eq-NP-cosym-Riem-cur-2}
\end{equation}
\begin{equation}
R(\partial,\xi,\xi,\overline{\partial}) = 0.  \label{eq-NP-cosym-Riem-cur-3}
\end{equation}
\end{cor}

Let $(M,g)$ be a Riemannian manifold. Then $M$ is said to be an Einstein
manifold if
\begin{equation}
S=a{g},  \label{eq-Einst}
\end{equation}
where $a$ is a real constant and $S$ is the Ricci tensor of $M$.

\begin{prop}
Let $\left(M,\varphi,\xi,\eta,g\right)$ be a $3$-dimensional trans-Sasakian
manifold of type $(\alpha,\beta)$ equipped with a Newman--Penrose frame
$(\partial,\overline{\partial},\xi)$. Then $M$ is an Einstein manifold given by
$(\ref{eq-Einst})$ if and only if
\begin{equation}
a = -|\rho_{\sc np} |^{2}-2|\beta_{\sc np} |^{2}-(\rho_{\sc np} -
\overline{\rho_{\sc np}})\epsilon_{\sc np} -\xi\left(\rho_{\sc np}\right)
-\rho_{\sc np}^{2}-\partial (\overline{\beta_{\sc np}})
-\overline{\partial}(\beta_{\sc np})= -2\xi \left(\rho_{\sc np} \right)
-2\rho_{\sc np}^{2},  \label{eq-NP-tr-Sas-Eins-1}
\end{equation}
\begin{equation}
\partial (\rho_{\sc np}) = 0.  \label{eq-NP-tr-Sas-Eins-2}
\end{equation}
\end{prop}

\noindent {\bf Proof.} From (\ref{eq-Einst}) and (\ref{eq-NPF-metric}), we have
\begin{equation}
S(\partial,\overline{\partial}) = S(\xi,\xi) =a,  \label{eq-NP-Einstein-1}
\end{equation}
\begin{equation}
S(\partial,\partial) = S(\partial,\xi)=0.  \label{eq-NP-Einstein-2}
\end{equation}
Using (\ref{eq-tsm-gen-Sach-2}), (\ref{eq-tsm-gen-Sach-3}), (\ref{eq-tsm-gen-Sach-1}) and (\ref{eq-tsm-gen-Sach-5}) in (\ref{eq-NP-Einstein-1}) and (\ref{eq-NP-Einstein-2}), we obtain (\ref{eq-NP-tr-Sas-Eins-1}) and (\ref{eq-NP-tr-Sas-Eins-2}), respectively. $\blacksquare$

\begin{cor}
Let $\left(M,\varphi,\xi,\eta,g\right)$ be a $3$-dimensional Sasakian manifold
equipped with a Newman--Penrose frame $(\partial,\overline{\partial},\xi)$.
Then $M$ is an Einstein manifold given by $(\ref{eq-Einst})$ if and only if
\begin{equation}
a=2=-2|\beta_{\sc np} |^{2}-2i\epsilon_{\sc np} - \partial(\overline{\beta_{\sc
np}}) -\overline{\partial}(\beta_{\sc np}).  \label{eq-NP-Sas-Eins}
\end{equation}
\end{cor}

\begin{cor}
Let $\left(M,\varphi,\xi,\eta,g\right)$ be a $3$-dimensional Kenmotsu manifold
equipped with a Newman--Penrose frame $(\partial,\overline{\partial},\xi)$.
Then $M$ is an Einstein manifold given by $(\ref{eq-Einst})$ if and only if
\begin{equation}
a=-2=-2-2|\beta_{\sc np} |^{2}-\partial (\overline{\beta_{\sc np}}) -
\overline{\partial}(\beta_{\sc np}).  \label{eq-NP-Ken-Eins}
\end{equation}
\end{cor}

\begin{cor}
Let $\left(M,\varphi,\xi,\eta,g\right)$ be a $3$-dimensional cosymplectic
manifold equipped with a Newman--Penrose frame
$(\partial,\overline{\partial},\xi)$. Then $M$ is an Einstein manifold given by
$(\ref{eq-Einst})$ if and only if
\begin{equation}
a=0=-2|\beta_{\sc np} |^{2}-\partial(\overline{\beta_{\sc np}}) -
\overline{\partial}(\beta_{\sc np}). \label{eq-NP-cosym-Eins}
\end{equation}
\end{cor}

\section{Rough Laplacian and geometric properties of vector
fields\label{Sect-NP-RL}}

Let $(M,g)$ be a $3$-dimensional Riemannian manifold, and let
$(e_{1},e_{2},\xi)$ be a real orthonormal frame on $M$. For any vector field
$X$ on $M$, the rough Laplacian $\Delta{X}$ of $X$ is defined by (see
\cite[Dragomir and Perrone 2012, p.~52]{Dragomir-Perrone-12-Book})
\begin{equation}
\Delta_{g} X = - \left\{\sum\limits_{i=1}^{2}(\nabla_{e_{i}}\nabla_{e_{i}}X
-\nabla_{\nabla_{e_{i}}e_{i}}X)+\nabla_{\xi}\nabla_{\xi}X
-\nabla_{\nabla_{\xi}\xi}X \right\}.  \label{eq-Lap-X}
\end{equation}

\begin{lem}
Let $(M,g)$ be a $3$-dimensional Riemannian manifold equipped with a
Newman--Penrose frame $(\partial,\overline{\partial},\xi)$. For any vector
field $X$ on $M$, the rough Laplacian of $X$ is given by
\begin{equation}
\Delta_{g} X = -\nabla_{\partial}\nabla_{\overline{\partial}}X
-\nabla_{\overline{\partial}}\nabla_{\partial}X
+\nabla_{\nabla_{\overline{\partial}}\partial}X
+\nabla_{\nabla_{\partial}\overline{\partial}}X
-\nabla_{\xi}\nabla_{\xi}X+\nabla_{\nabla_{\xi}\xi}X.  \label{eq-Lap-X-NP}
\end{equation}
\end{lem}

\noindent {\bf Proof.} From (\ref{eq-partial-partial-bar}), we have
\begin{equation}
e_{1}=\frac{1}{\sqrt{2}}(\partial +\overline{\partial}), \qquad
e_{2}=\frac{i}{\sqrt{2}}(\partial -\overline{\partial}).  \label{eq-Lap-X-NP-1}
\end{equation}
Using (\ref{eq-Lap-X-NP-1}) in (\ref{eq-Lap-X}), we obtain (\ref{eq-Lap-X-NP}).
$\blacksquare$

\begin{lem}
Let $(M,g)$ be a $3$-dimensional Riemannian manifold equipped with a
Newman--Penrose frame $(\partial,\overline{\partial},\xi)$. Then the rough
Laplacian of the vector field $\xi$ is given by
\begin{eqnarray}
\Delta_{g} \xi &=& \left(-\overline{\partial} (\overline{\rho_{\sc np}})
+\partial(\overline{\sigma_{\sc np}})+2\overline{\sigma_{\sc np}}\beta_{\sc np}
+\rho_{\sc np} \overline{\kappa_{\sc np}}+\xi (\overline{\kappa_{\sc np}})
+\overline{\kappa_{\sc np}}\epsilon_{\sc np} +\kappa_{\sc np}
\overline{\sigma_{\sc np}}\right) \partial\nonumber \\
&&+\left(-\partial(\rho_{\sc np})+\overline{\partial}(\sigma_{\sc np})
+2\sigma_{\sc np} \overline{\beta_{\sc np}}+\overline{\rho_{\sc np}}\kappa_{\sc np}
+\xi (\kappa_{\sc np})+\kappa_{\sc np} \overline{\epsilon_{\sc np}}
+\overline{\kappa_{\sc np}}\sigma_{\sc np} \right) \overline{\partial}  \nonumber \\
&&+2\left(\left\vert \kappa_{\sc np} \right\vert^{2} + \left\vert\rho_{\sc
np}\right\vert^{2}+\left\vert \sigma_{\sc np} \right\vert^{2}\right) \xi .
\label{eq-Laplacian-NP-xi}
\end{eqnarray}
\end{lem}

\noindent {\bf Proof.} Putting $X=\xi$ in (\ref{eq-Lap-X-NP}), we get
\begin{equation}
\Delta_{g} \xi = -\nabla_{\partial}\nabla_{\overline{\partial}}\xi
-\nabla_{\overline{\partial}}\nabla_{\partial}\xi
+\nabla_{\nabla_{\overline{\partial}}\partial}\xi
+\nabla_{\nabla_{\partial}\overline{\partial}}\xi
-\nabla_{\xi}\nabla_{\xi}\xi+\nabla_{\nabla_{\xi}\xi}\xi.
\label{eq-Laplacian-NP-xi-1}
\end{equation}
Using (\ref{eq-PNF-nabla(partial,partial)}),
(\ref{eq-PNF-nabla(partial,barpartial)}), (\ref{eq-PNF-nabla(partial,xi)}),
(\ref{eq-PNF-nabla(xi,partial)}) and (\ref{eq-PNF-nabla(xi,xi)}) in
(\ref{eq-Laplacian-NP-xi-1}), we get (\ref{eq-Laplacian-NP-xi}). $\blacksquare$

\begin{thm}
Let $(M,\varphi,\xi,\eta,g)$ be a $3$-dimensional trans-Sasakian manifold of
type $(\alpha,\beta)$ equipped with a Newman--Penrose frame
$(\partial,\overline{\partial},\xi)$. Then the rough Laplacian of the vector
field $\xi$ is given by
\begin{equation}
\Delta_{g} \xi = - \overline{\partial}(\overline{\rho_{\sc np}})\partial
-\partial(\rho_{\sc np})\overline{\partial}+2\left\vert \rho_{\sc np}
\right\vert^{2}\xi . \label{eq-NP-RL-tsm}
\end{equation}
\end{thm}

\noindent {\bf Proof.} Using (\ref{eq-tsm-NP}) in (\ref{eq-Laplacian-NP-xi}),
we get (\ref{eq-NP-RL-tsm}). $\blacksquare$

Let $(M,g)$ be a $3$-dimensional Riemannian manifold equipped with a
Newman--Penrose frame $(\partial,\overline{\partial},\xi)$. For the vector
field $\xi$, the standard Bochner formula (see \cite[Dragomir and Perrone 2012,
p.~53]{Dragomir-Perrone-12-Book}) is given by
\begin{equation}
g\left(\Delta_{g}\xi,\xi \right)=\frac{1}{2}\Delta \left\Vert\xi
\right\Vert^{2}+\left\Vert \nabla \xi \right\Vert^{2},  \label{eq-Bochner-for}
\end{equation}
where $\Delta$ is the ordinary Laplace--Beltrami operator on functions. Since
$\left\Vert\xi \right\Vert =1$, therefore (\ref{eq-Bochner-for}) becomes
\begin{equation}
\left\Vert \nabla \xi \right\Vert^{2}=g\left(\Delta_{g} \xi,\xi \right).
\label{eq-Bochner-for-1}
\end{equation}

\begin{lem} \label{lem-covariant-xi}
Let $(M,g)$ be a $3$-dimensional Riemannian manifold equipped with a
Newman--Penrose frame $(\partial,\overline{\partial},\xi)$. Then,
\begin{equation}
\left\Vert \nabla \xi \right\Vert^{2}=2\left(\left\vert \kappa_{\sc np}
\right\vert^{2}+\left\vert \rho_{\sc np} \right\vert^{2} + \left\vert
\sigma_{\sc np} \right\vert^{2}\right).  \label{eq-covariant-der-xi}
\end{equation}
\end{lem}

\noindent {\bf Proof.} Using (\ref{eq-Laplacian-NP-xi}) in
(\ref{eq-Bochner-for-1}), we get (\ref{eq-covariant-der-xi}). $\blacksquare$

\begin{thm}
Let $(M,\varphi,\xi,\eta,g)$ be a $3$-dimensional trans-Sasakian manifold of
type $(\alpha,\beta)$ equipped with a Newman--Penrose frame
$(\partial,\overline{\partial},\xi)$. Then,
\begin{equation}
\left\Vert \nabla \xi \right\Vert^{2} = 2\left\vert \rho_{\sc np}
\right\vert^{2}. \label{eq-covariant-der-xi-tsm}
\end{equation}
\end{thm}

\noindent {\bf Proof.} Using (\ref{eq-tsm-NP}) in (\ref{eq-covariant-der-xi}),
we get (\ref{eq-covariant-der-xi-tsm}). $\blacksquare$

Let $(M,g)$ be a Riemannian manifold. A non-zero vector field $X$ on $M$ is
said to be pointwise collinear with its rough Laplacian if there exists a
smooth function $\lambda$ on $M$ such that
\begin{equation}
\Delta_{g} X=\lambda X.  \label{eq-eigenvector}
\end{equation}

\begin{thm}
Let $(M,\varphi,\xi,\eta,g)$ be a $3$-dimensional trans-Sasakian manifold of
type $(\alpha,\beta)$  equipped with a Newman--Penrose frame
$(\partial,\overline{\partial},\xi)$. Then the structure vector field $\xi$ is
pointwise collinear with its rough Laplacian, that is,
\begin{equation}
\Delta_{g} \xi = \lambda \xi  \label{eq-eigenvector-xi}
\end{equation}
for some smooth function $\lambda$, if and only if
\begin{equation}
\partial (\rho_{\sc np})=0. \label{label-point-colli-cond}
\end{equation}
Moreover, the corresponding function $\lambda$ is given by
\begin{equation}
\lambda = 2\left\vert \rho_{\sc np} \right\vert ^{2}.  \label{eq-EF-tsm-Lap}
\end{equation}
\end{thm}

\noindent {\bf Proof.} Using (\ref{eq-NP-RL-tsm}) in (\ref{eq-eigenvector-xi}),
we obtain
\begin{equation}
\lambda \xi = -\overline{\partial}(\overline{\rho_{\sc np}})\partial -\partial
(\rho_{\sc np})\overline{\partial} +2\left\vert \rho_{\sc np}
\right\vert^{2}\xi. \label{eq-eigenvector-xi-1}
\end{equation}
Since the frame $(\partial,\overline{\partial},\xi)$ is linearly independent,
it follows from (\ref{eq-eigenvector-xi-1}) that
\begin{equation}
\overline{\partial}(\overline{\rho_{\sc np}})=0,\quad \partial (\rho_{\sc np})=0,
\label{eq-eigenvector-xi-2}
\end{equation}
\begin{equation}
\lambda - 2\left\vert \rho_{\sc np} \right\vert ^{2}=0.
\label{eq-eigenvector-xi-3}
\end{equation}
Hence, in view of (\ref{eq-eigenvector-xi-2}), we can easily see that the
structure vector field $\xi$ is pointwise collinear with its rough Laplacian if
and only if (\ref{label-point-colli-cond}) is true. Moreover,
(\ref{eq-EF-tsm-Lap}) follows from (\ref{eq-eigenvector-xi-3}) directly.
$\blacksquare$

\begin{cor}
Let $(M,\varphi,\xi,\eta,g)$ be a $3$-dimensional Sasakian, Kenmotsu, or
cosymplectic manifold equipped with a Newman--Penrose frame
$(\partial,\overline{\partial},\xi)$. Then the structure vector field $\xi$ is
pointwise collinear with its rough Laplacian. Moreover, in case of Sasakian or
Kenmotsu, $\Delta_g \xi = 2\xi$; and in case of cosymplectic, $\Delta_g \xi
=0$.
\end{cor}

Let $(M,g)$ be a Riemannian manifold. A vector field $X$ on $M$ is said to be
parallel \cite[Petersen 2006, p.~28]{Petersen-06-Book} if
\begin{equation}
\nabla X=0.  \label{eq-parallel-vector-field}
\end{equation}

\begin{thm}\label{th-RM-xi-Par}
Let $(M,g)$ be a $3$-dimensional Riemannian manifold equipped with a
Newman--Penrose frame $(\partial,\overline{\partial},\xi)$. Then the structure
vector field $\xi$ is parallel if and only if $\xi$ generates a geodesic,
shear-free, non-expanding, and twist-free congruence.
\end{thm}

\noindent {\bf Proof.} In view of Lemma~\ref{lem-covariant-xi} and
(\ref{eq-parallel-vector-field}), the structure vector field $\xi$ is parallel
if and only if
\begin{equation}
\kappa_{\sc np}=\rho_{\sc np}=\sigma_{\sc np}=0.
\end{equation}
In particular, $\kappa_{\sc np} = 0$ implies that $\xi$ is geodesic,
$\sigma_{\sc np} = 0$ implies that $\xi$ is shear-free, and $\rho_{\sc np} = 0$
implies that $\xi$ is both non-expanding and twist-free. $\blacksquare$

\begin{thm}\label{th-cosym-equi}
Let $(M,\varphi,\xi,\eta,g)$ be a $3$-dimensional trans-Sasakian manifold
equipped with a Newman--Penrose frame $(\partial,\overline{\partial},\xi)$.
Then the structure vector field $\xi$ is parallel if and only if $M$ is
cosymplectic.
\end{thm}

\noindent {\bf Proof.} In view of Remark~(\ref{rem-tS-cases}), and
Theorem~\ref{th-RM-xi-Par}, we get the result. $\blacksquare$

Now, we present the following example to illustrate
Theorem~\ref{th-cosym-equi}.
\begin{ex-new}
Let
\[
M=\{(x^{1},x^{2},x^{3})\in {\Bbb R}^{3}\}
\]
be a $3$-dimensional smooth manifold with standard coordinates
$(x^{1},x^{2},x^{3})$. Consider a global frame $(B_{1},B_{2},B_{3})$ on $M$
given by
\[
B_{1} = \frac{\partial}{\partial x^{1}},\quad B_{2} = \frac{\partial}{\partial
x^{2}},\quad B_{3} = \frac{\partial}{\partial x^{3}}.
\]
Define the associated Newman--Penrose frame
$(\partial,\overline{\partial},\xi)$ by
\[
\partial = \frac{1}{\sqrt{2}}(B_{1}-iB_{2}),
\quad \overline{\partial}=\frac{1}{\sqrt{2}}(B_{1}+iB_{2}),\quad \xi
=B_{3}.
\]
Define the $1$-form $\eta $ by $\eta (X)=g(X,B_{3})$ for every vector field $X$
on $M$. Let $\varphi $ be the $(1,1)$-tensor field defined by
\[
\varphi (\partial) = i\partial,\quad \varphi
(\overline{\partial})=-i\overline{\partial},\quad \varphi (\xi)=0.
\]
Define the Riemannian metric $g$ by
\[
g(\xi,\xi) = g(\partial,\overline{\partial})=1,\quad
g(\xi,\partial)=g(\partial,\partial) =
g(\overline{\partial},\overline{\partial})=0.
\]
Let $\nabla$ be the Levi-Civita connection of $g$. Then, by a direct
calculation, we obtain
\[
[B_{1},B_{2}]=[B_{1},B_{3}]=[B_{2},B_{3}]=0,
\]
\[
\nabla_{B_{1}}\xi=\nabla_{B_{2}}\xi=\nabla_{B_{3}}\xi=0,
\]
\[
\nabla_{\partial}\xi=\nabla_{\overline{\partial}}\xi=0.
\]
In particular,
\[
\nabla_{B_{i}}B_{j}=0\quad {\rm for\;all\;}i,j,
\]
that is, $\xi$ is parallel. Using the above computation in
(\ref{eq-spin-coe-1}), we obtain
\[
\kappa_{\sc np} = -g(\partial,\nabla_{\xi}\xi) = 0, \quad \sigma_{\sc np} = -
g(\partial,\nabla_{\partial}\xi) = 0,\quad \rho_{\sc np} =
g(\partial,\nabla_{\overline{\partial}}\xi) = 0.
\]
Hence $(M,\varphi,\xi,\eta,g)$ is a trans-Sasakian manifold of type
$(\alpha,\beta)= (0,0)$. Consequently, it is a cosymplectic manifold. In this
case, integral curves of $\xi$ generate a congruence that is geodesic,
shear-free, non-expanding, and twist-free.
\end{ex-new}

Let $(M,g)$ be a $3$-dimensional Riemannian manifold equipped with a
Newman-Penrose frame $(\partial,\overline{\partial},\xi)$. Then the structure
vector field $\xi $ is called harmonic vector field if (see \cite[Dragomir and
Perrone 2012]{Dragomir-Perrone-12-Book})
\begin{equation}
\Delta_{g}\xi =\left\Vert \nabla \xi \right\Vert^{2}\xi. \label{eq-xi-harmonic}
\end{equation}

\begin{lem} \label{lem-harmonic-xi}
Let $(M,g)$ be a $3$-dimensional Riemannian manifold equipped with a
Newman--Penrose frame $(\partial,\overline{\partial},\xi)$. Then the structure
vector field $\xi$ is harmonic if and only if
\begin{equation}
\partial(\rho_{\sc np})-\overline{\partial}(\sigma_{\sc np})-2\sigma_{\sc np}
\overline{\beta_{\sc np}}-\overline{\rho_{\sc np}}\kappa_{\sc np} - \xi
(\kappa_{\sc np})-\kappa_{\sc np} \overline{\epsilon_{\sc
np}}-\overline{\kappa_{\sc np}}\sigma_{\sc np}=0.\label{eq-harmonic-RM}
\end{equation}
\end{lem}

\noindent {\bf Proof.} Using (\ref{eq-Laplacian-NP-xi}) and
(\ref{eq-covariant-der-xi}) in (\ref{eq-xi-harmonic}), we obtain
\begin{eqnarray}
0&=&\left(\overline{\partial}(\overline{\rho_{\sc np}})
-\partial(\overline{\sigma_{\sc np}})-2\overline{\sigma_{\sc np}}\beta_{\sc np}
-\rho_{\sc np} \overline{\kappa_{\sc np}}-\xi (\overline{\kappa_{\sc np}})
-\overline{\kappa_{\sc np}}\epsilon_{\sc np} -\kappa_{\sc np}
\overline{\sigma_{\sc np}}\right) \partial\nonumber \\
&&+\left(\partial(\rho_{\sc np})-\overline{\partial}(\sigma_{\sc np})
-2\sigma_{\sc np} \overline{\beta_{\sc np}}-\overline{\rho_{\sc np}}\kappa_{\sc
np} -\xi (\kappa_{\sc np})-\kappa_{\sc np} \overline{\epsilon_{\sc
np}}-\overline{\kappa_{\sc np}}\sigma_{\sc np} \right) \overline{\partial}.
\label{eq-harmonic-RM-1}
\end{eqnarray}
Since the frame $(\partial,\overline{\partial},\xi)$ is linearly independent,
it follows from (\ref{eq-harmonic-RM-1}) that
\begin{eqnarray}
&&\overline{\partial}(\overline{\rho_{\sc np}}) -\partial(\overline{\sigma_{\sc np}})
-2\overline{\sigma_{\sc np}}\beta_{\sc np} -\rho_{\sc np} \overline{\kappa_{\sc np}}
-\xi (\overline{\kappa_{\sc np}})-\overline{\kappa_{\sc np}}\epsilon_{\sc np}
-\kappa_{\sc np} \overline{\sigma_{\sc np}}=0, \label{eq-harmonic-RM-2}\\
&&\partial(\rho_{\sc np})-\overline{\partial}(\sigma_{\sc np}) -2\sigma_{\sc
np} \overline{\beta_{\sc np}}-\overline{\rho_{\sc np}}\kappa_{\sc np} - \xi
(\kappa_{\sc np})-\kappa_{\sc np} \overline{\epsilon_{\sc
np}}-\overline{\kappa_{\sc np}}\sigma_{\sc np}=0.\label{eq-harmonic-RM-3}
\end{eqnarray}
Observing that (\ref{eq-harmonic-RM-2}) is the complex conjugate of
(\ref{eq-harmonic-RM-3}), it suffices to consider (\ref{eq-harmonic-RM-3}),
which yields the result. $\blacksquare$

\begin{thm}\label{th-harmonic-xi-tsm}
Let $(M,\varphi,\xi,\eta,g)$ be a $3$-dimensional trans-Sasakian manifold
equipped with a Newman--Penrose frame $(\partial,\overline{\partial},\xi)$.
Then the structure vector field $\xi$ is harmonic if and only if
\begin{equation}
\partial(\rho_{\sc np})=0. \label{eq-harmonic-tsm}
\end{equation}
\end{thm}

\noindent {\bf Proof.} In view of Lemma~\ref{lem-harmonic-xi} and
(\ref{eq-tsm-NP}), we get (\ref{eq-harmonic-tsm}). $\blacksquare$

\begin{cor}
Let $(M,\varphi,\xi,\eta,g)$ be a $3$-dimensional Sasakian, Kenmotsu or
cosymplectic manifold equipped with a Newman--Penrose frame
$(\partial,\overline{\partial},\xi)$. Then the structure vector field $\xi$ is
harmonic.
\end{cor}

Let $(M,g)$ be a $3$-dimensional Riemannian manifold. Let $(e_{1},e_{2},e_{3})$
be a local orthonormal frame on $M$. The divergence of a vector field $X$ is
defined by (see \cite[Petersen 2006, p.~28]{Petersen-06-Book})
\begin{equation}
{\rm div}(X) = \sum_{i=1}^{3}g\left(\nabla_{e_{i}}X,e_{i}\right).\label{eq-div}
\end{equation}

\begin{lem}
\label{lem-div-xi-NP} Let $(M,g)$ be a $3$-dimensional Riemannian manifold
equipped with a Newman--Penrose frame $(\partial,\overline{\partial},\xi)$.
Then,
\begin{equation}
{\rm div}(\xi) = 2\Theta_{\sc np}.  \label{eq-div-NP}
\end{equation}
\end{lem}

\noindent {\bf Proof.} From (\ref{eq-div}), we have
\begin{equation}
{\rm div}(\xi) = g\left(\nabla_{e_{1}}\xi,e_{1}\right) +
g\left(\nabla_{e_{2}}\xi,e_{2}\right) +g\left(\nabla_{\xi}\xi,\xi \right).
\label{eq-div-NP-1}
\end{equation}
Using (\ref{eq-Lap-X-NP-1}) in (\ref{eq-div-NP-1}), we get
\begin{equation}
{\rm div}(\xi) = g\left(\nabla_{\overline{\partial}}\xi,\partial \right) +
\left(\nabla_\partial \xi,{\overline{\partial}}\right).  \label{eq-Lap-X-NP-2}
\end{equation}
Using (\ref{eq-spin-coe-1}) in (\ref{eq-Lap-X-NP-2}), we obtain
\begin{equation}
{\rm div}(\xi) = \rho_{\sc np} + \overline{\rho_{\sc np}}.
\label{eq-Lap-X-NP-3}
\end{equation}
In view of (\ref{eq-rho}) and (\ref{eq-Lap-X-NP-3}), we obtain
(\ref{eq-div-NP}). $\blacksquare$

\begin{lem} \label{lem-div-free-xi}
Let $(M,g)$ be a $3$-dimensional Riemannian manifold equipped with a
Newman--Penrose frame $(\partial,\overline{\partial},\xi)$. Then the vector
field $\xi$ is divergence-free if and only if the congruence generated by $\xi$
is non-expanding, that is, $\Theta_{\sc np} =0$. Moreover, the congruence
generated by $\xi$ is expanding (respectively, contracting) if and only if
$\Theta_{\sc np} > 0$ (respectively, $\Theta_{\sc np} < 0$).
\end{lem}

\noindent {\bf Proof.} In view of Lemma~\ref{lem-div-xi-NP}, the vector field
$\xi$ is divergence-free if and only if $\Theta_{\sc np} =0$. Moreover, the
congruence generated by $\xi$ is expanding (respectively, contracting) if and
only if ${\rm div}(\xi)>0$ (respectively, ${\rm div}(\xi)<0$). By
Lemma~\ref{lem-div-xi-NP}, this is equivalent to $\Theta_{\sc np} > 0$
(respectively, $\Theta_{\sc np} < 0$). $\blacksquare$

\begin{thm}\label{th-tsm-equi}
Let $(M,\varphi,\xi,\eta,g)$ be a $3$-dimensional trans-Sasakian manifold of
type $(\alpha,\beta)$ equipped with a Newman--Penrose frame
$(\partial,\overline{\partial},\xi)$. Then the following are equivalent.
\begin{itemize}
\item[{\bf (a)}] The vector field $\xi$ is divergence-free.

\item[{\bf (b)}] The congruence generated by $\xi$ is non-expanding.

\item[{\bf (c)}] $M$ is a ${\cal C}_{6}$-manifold.
\end{itemize}
\end{thm}

\noindent {\bf Proof.} The proof follows from Lemma~\ref{lem-div-free-xi} and
Remark~\ref{rem-tS-cases}. $\blacksquare$

Now, we present the following example to illustrate Theorem~\ref{th-tsm-equi}.
\begin{ex-new}
Let
\[
M=\{(x^{1},x^{2},x^{3})\in {\mathbb R}^{3}: x^{1}>0\}
\]
be a $3$-dimensional smooth manifold with standard coordinates
$(x^{1},x^{2},x^{3})$. Define a linearly independent global frame
$(B_{1},B_{2},B_{3})$ on $M$ by
\[
B_{1}=\frac{\partial}{\partial x^{1}} -
\frac{2x^{2}}{x^{1}}\frac{\partial}{\partial x^{3}},\quad
B_{2}=\frac{\partial}{\partial x^{2}}+2\frac{\partial}{\partial x^{3}},\quad
B_{3}=\frac{\partial}{\partial x^{3}}.
\]
Define the associated Newman--Penrose frame
$(\partial,\overline{\partial},\xi)$ by
\[
\partial=\frac{1}{\sqrt{2}}(B_{1}-iB_{2}),\quad
\overline{\partial}=\frac{1}{\sqrt{2}}(B_{1}+iB_{2}),\quad
\xi=B_{3}.
\]
Define the $1$-form $\eta$ on $M$ by
\[
\eta(X)=g(X,B_{3})
\]
for every vector field $X$ on $M$. Let $\varphi$ be the $(1,1)$ tensor field
defined by
\[
\varphi(\partial)=i\partial,\quad
\varphi(\overline{\partial})=-i\overline{\partial},\quad \varphi\xi=0.
\]
Consider the Riemannian metric $g$ defined by
\[
g(\xi,\xi)=g(\partial,\overline{\partial})=1,\qquad
g(\xi,\partial)=g(\partial,\partial)=g(\overline{\partial},\overline{\partial})=0.
\]
Let $\nabla$ denote the Levi-Civita connection of $g$. Then, by a direct
calculation, we obtain
\[
[B_{1},B_{2}]= \frac{2}{x^{1}}B_{3}, \qquad [B_{1},B_{3}]=0, \qquad
[B_{2},B_{3}]=0,
\]
\[
\nabla_{B_{1}}\xi =-\frac{1}{x^{1}}B_{2},\qquad \nabla_{B_{2}}\xi =
\frac{1}{x^{1}} B_{1},\qquad \nabla_{\xi}\xi =0,
\]
\[
\nabla_{\partial}\xi =- \frac{i}{x^{1}}\partial,\qquad
\nabla_{\overline{\partial}}\xi= \frac{i}{x^{1}}\overline{\partial}.
\]
Using the above computation in (\ref{eq-spin-coe-1}), we obtain
\[
\kappa_{\sc np} =-g(\partial,\nabla_{\xi}\xi)=0,\qquad \sigma_{\sc np} =
-g(\partial,\nabla_{\partial}\xi)=0,\qquad \rho_{\sc np} =
g(\partial,\nabla_{\overline{\partial}}\xi)= \frac{i}{x^{1}}.
\]
Thus, $(\varphi,\xi,\eta,g)$ defines a trans-Sasakian structure on $M$ of type
$(\alpha,\beta)=\left(\frac{1}{x^{1}},0\right)$. Consequently, it is a ${\cal
C}_{6}$-manifold. In this case, integral curves of $\xi$ generate a congruence
which is geodesic, shear-free and non-expanding. Moreover, the expansion $\Theta_{\sc np}$ and twist $\omega_{\sc np}$ are given by
\[
\Theta_{\sc np} = 0,\quad \omega_{\sc np} = \frac{1}{x^{1}}.
\]
Furthermore, from (\ref{eq-div-NP}), it follows that ${\rm div}(\xi)=0$, and
hence the vector field $\xi$ is divergence-free.
\end{ex-new}

\begin{cor}
Let $(M,\varphi,\xi,\eta,g)$ be a $3$-dimensional trans-Sasakian manifold of
type $(\alpha,\beta)$ equipped with a Newman--Penrose frame
$(\partial,\overline{\partial},\xi)$. Then the following hold:
\begin{itemize}
\item[{\bf (a)}] If $M$ is a Sasakian manifold, then $\xi$ is divergence-free.

\item[{\bf (b)}] If $M$ is a Kenmotsu manifold, then ${\rm div}(\xi)=2$.

\item[{\bf (c)}] If $M$ is a cosymplectic manifold, then $\xi$ is
divergence-free.
\end{itemize}
\end{cor}

\section{Newman--Penrose computation for trans-Sasakian structures on
homogeneous metrics of type ${\Bbb E}(\kappa,\tau)$\label{Sect-NP-homo-met}}

The spaces ${\Bbb E}(\kappa,\tau)$ form the standard two-parameter family of
simply connected, oriented, homogeneous Riemannian $3$-manifolds with a
$4$-dimensional isometry group. Geometrically, ${\Bbb E}(\kappa,\tau)$ is
characterized by a Riemannian submersion
\[
\pi:{\Bbb E}(\kappa,\tau)\longrightarrow {M}^2(\kappa),
\]
where ${M}^2(\kappa)$ is the simply connected surface of constant curvature
$\kappa$, whose fibres are the integral curves of a unit vertical Killing
vector field $\xi$; the constant $\tau$ is the bundle curvature of this
fibration. The condition $\kappa\neq 4\tau^2$ gives the non-space-form members
of the family. In particular, when $\tau=0$, one obtains the product spaces
${\Bbb S}^2(\kappa)\times{\Bbb R}$ and ${\Bbb H}^2(\kappa)\times{\Bbb R}$,
while for $\tau\neq0$, the family includes important examples such as
Heisenberg group ${\rm Nil}_3$, Berger spheres, and the universal cover of
${\rm PSL}_2({\Bbb R})$ (see \cite[Daniel 2007]{Daniel-07-CMH}). The purpose of
this section is to present, in a completely explicit and verifiable form, the
computation showing how the Newman--Penrose formalism can be used on the
homogeneous metrics of type ${\Bbb E}(\kappa,\tau)$.  The main conclusions are:
\begin{itemize}
\item[\bf (a)] if $\tau\neq 0$ and $\kappa\neq 4\tau^2$, every compatible
trans-Sasakian structure is vertical, hence $\alpha$-Sasakian; \item[\bf (b)]
if $\tau=0$ and $\kappa\neq 0$, every compatible trans-Sasakian structure is
vertical, hence cosymplectic.
\end{itemize}
Consequently, in the genuine non-space-form geometries of type ${\Bbb
E}(\kappa,\tau)$, there are no proper trans-Sasakian structures compatible with
the homogeneous metric. The exceptional cases $\kappa=4\tau^2$ and
$\kappa=\tau=0$ are precisely the space-form/flat cases not ruled out by this
argument.

Let $(M,\varphi,\xi,\eta,g)$ be a trans-Sasakian manifold of type
$(\alpha,\beta)$. Since $\xi$ is a unit vector field, we have
$g(\nabla_X\xi,\xi)=0$ for any vector field $X$ on $M$, hence
\begin{equation}
\nabla_{e_1}\xi = f_{11}e_1+f_{12}e_2, \qquad
\nabla_{e_2}\xi=f_{21}e_1+f_{22}e_2, \qquad \nabla_\xi\xi=f_{31}e_1+f_{32}e_2
\label{eq:ABCDPQ}
\end{equation}
for some real-valued functions $f_{11}$, $f_{12}$, $f_{21}$, $f_{22}$,
$f_{31}$, $f_{32}$. Then, using (\ref{eq-partial-partial-bar}), we get
\[
\nabla_\partial\xi = \frac1{\sqrt2}\bigl((f_{11} - if_{21})e_1+(f_{12} -
if_{22})e_2\bigr), \qquad \nabla_{\overline\partial}\xi
=\frac1{\sqrt2}\bigl((f_{11}+if_{21})e_1+(f_{12}+if_{22})e_2\bigr).
\]
Hence from (\ref{eq-spin-coe-1}), we obtain
\begin{equation}\label{eq:master-np-formulas}
\kappa_{\sc np} = -\frac{1}{\sqrt2}(f_{31}-if_{32}), \quad \sigma_{\sc
np}=-\frac12\bigl((f_{11}-f_{22})-i(f_{21}+f_{12})\bigr), \quad \rho_{\sc
np}=\frac12\bigl((f_{11}+f_{22})+i(f_{21}-f_{12})\bigr).
\end{equation}
In particular, the trans-Sasakian condition (\ref{eq-tsm-NP}) is equivalent to
\begin{equation}\label{eq:trans-condition-ABCD}
f_{31}=f_{32}=0, \qquad f_{11}=f_{22}, \qquad f_{21}=-f_{12}.
\end{equation}
In that case
\begin{equation}\label{eq:rho-alpha-beta-master}
\rho_{\sc np}=f_{11}+if_{21}.
\end{equation}
Thus, with our conventions
\[
\rho_{\sc np}=\beta+i\alpha, \qquad \beta=f_{11}=f_{22}, \qquad
\alpha=f_{21}=-f_{12}.
\]

\subsection{The non-product case $\tau\neq 0$}

Let $(M^{3},g)={\Bbb E}(\kappa,\tau)$, equipped with its standard homogeneous
metric $g$, with $\tau\neq 0$. We use the standard global homogeneous
orthonormal frame $(E_1,E_2,E_3)$, normalized as in \cite[Daniel 2007,
Section~2.1]{Daniel-07-CMH}, where $E_3$ is the canonical vertical unit Killing
vector field associated with the fibration
\[
{\Bbb E}(\kappa,\tau)\longrightarrow {M}^2(\kappa).
\]
In Daniel's notation, the bracket constant is $\sigma=\kappa/(2\tau)$; we write
\begin{equation}
\delta:=\sigma-2\tau=\frac{\kappa-4\tau^2}{2\tau}. \label{eq-delta}
\end{equation}
With this convention, the Lie bracket relations are
\begin{equation}
[E_1,E_2]=2\tau E_3,\qquad [E_2,E_3]=\sigma E_1,\qquad [E_3,E_1]=\sigma E_2,
\label{eq-Lie-bracket-non-pro}
\end{equation}
and the Levi-Civita connection $\nabla$ is expressed by
\begin{align}
\nabla_{E_1}E_1&=0, & \nabla_{E_1}E_2&=\tau E_3, & \nabla_{E_1}E_3&=-\tau E_2, \notag \\
\nabla_{E_2}E_1&=-\tau E_3, & \nabla_{E_2}E_2&=0, & \nabla_{E_2}E_3&=\tau E_1,
\label{eq:nabla-background-detailed} \\
\nabla_{E_3}E_1&=(\sigma-\tau)E_2, &
\nabla_{E_3}E_2&=-(\sigma-\tau)E_1, & \nabla_{E_3}E_3&=0. \notag
\end{align}
These are precisely the canonical-frame formulas \cite[Daniel 2007,
Section~2.1]{Daniel-07-CMH}.

We now let $\xi$ denote the structure vector field of a hypothetical
trans-Sasakian structure on $M$. We shall show that, in the genuine non-product
case, $\xi$ is forced to be vertical, namely $\xi=\pm E_3$. Here ``vertical''
means tangent to the fibres of the above fibration. It suffices to analyse the
open set on which $\xi$ has a nonzero horizontal component and to prove that
this open set is empty. Let
\[
U=\{p\in M:\xi(p)\notin\{\pm E_3(p)\}\}.
\]
On $U$, the horizontal component of $\xi$ is nonzero. Thus, if
\[
\xi = \xi^{\rm{hor}}+g(\xi,E_3)E_3, \qquad \xi^{\rm{hor}}\in
\operatorname{span}\{E_1,E_2\},
\]
we define the unit horizontal vector field $T:=\frac{\xi^{\rm{hor}}}{\Vert
\xi^{\rm{hor}}\Vert}$. Since $\xi$ is a unit vector field, there is a local
function $u$ such that $g(\xi,E_3)=\cos u$, $\Vert\xi^{\mathrm{hor}}\Vert=\sin
u$. Moreover, since $T$ is a unit vector field in the horizontal plane, there
is a local function $v$ such that
\[
T=\cos v\,E_1+\sin v\,E_2.
\]
Hence, locally on $U$, we can write
\begin{equation}
\xi =\sin u\,T+\cos u\,E_3, \label{eq:xi-angles-detailed}
\end{equation}
\begin{equation}
T=\cos v\,E_1+\sin v\,E_2. \label{eq:T-angles-detailed}
\end{equation}

To compute the Newman--Penrose coefficients of the candidate structure vector
field $\xi$, we need an orthonormal frame whose third vector field is $\xi$. We
choose this frame so as to remain adapted to the horizontal--vertical splitting
of ${\Bbb E}(\kappa,\tau)$, i.e. where $e_1$ is the horizontal unit vector
field orthogonal to the horizontal direction $T$, and $e_2$ is the unit vector
field in the $(T,E_3)$-plane orthogonal to $\xi$.
\begin{equation}
e_1 =-\sin v\,E_1 + \cos v\,E_2, \label{eq:adapted-frame-detailed(e1)}
\end{equation}
\begin{equation}
e_2=-\cos u\,T + \sin u\,E_3. \label{eq:adapted-frame-detailed(e2)}
\end{equation}
Note that the inverse relations are
\begin{equation}
T=\sin u\,\xi-\cos u\,e_2, \label{eq:inverse-frame-detailed(T)}
\end{equation}
\begin{equation}
E_3=\cos u\,\xi+\sin u\,e_2. \label{eq:inverse-frame-detailed(E3)}
\end{equation}
In particular,
\begin{equation}
\cos u\,T-\sin u\,E_3=-e_2. \label{eq:key-sign-detailed}
\end{equation}
We now compute the covariant derivatives for the adapted frame.

\begin{lem}\label{lem:adapted-covariant-derivatives}
With the notation introduced above, define
\begin{equation}
C(X):= {\rm d}v(X)+(\sigma-\tau)g(X,E_3) \label{eq:C(X)}
\end{equation}
for any vector field $X$ on $M$. Then the covariant derivatives of the
horizontal vector fields $T$ and $e_{1}$ are given by
\begin{align}
\nabla_XT &=C(X)e_1-\tau g(X,e_1) E_3, \label{eq:nabla-T-detailed}\\
\nabla_Xe_1 &=-C(X)T+\tau g(X,T) E_3. \label{eq:nabla-e1-detailed}
\end{align}
Moreover,
\begin{equation}
\nabla_{e_1}E_3 =\tau T, \qquad \nabla_{e_2}E_3 =\tau\cos u\,e_1, \qquad
\nabla_{\xi}E_3 =-\tau\sin u\,e_1, \label{eq:nabla-E3-special}
\end{equation}
and the covariant derivatives of $\xi$ in the adapted frame $(e_1,e_2,\xi)$ are
\begin{align}
\nabla_{e_1}\xi &=\sin u\,{\rm d}v(e_1)e_1-\bigl({\rm d}u(e_1)+\tau\bigr)e_2,
\label{eq:nabla-e1-xi-detailed}\\
\nabla_{e_2}\xi &=\bigl(\sin u\,{\rm d}v(e_2)+\tau+\delta\sin^2u\bigr)e_1-{\rm
d}u(e_2)e_2, \label{eq:nabla-e2-xi-detailed}\\
\nabla_{\xi}\xi &=\sin u\bigl({\rm d}v(\xi)+\delta\cos u\bigr)e_1- {\rm
d}u(\xi)e_2. \label{eq:nabla-xi-xi-detailed}
\end{align}
\end{lem}

\noindent{\bf Proof.}
Let
\[
X=x_1E_1+x_2E_2+x_3E_3.
\]
From the connection formulas (\ref{eq:nabla-background-detailed}), we have
\begin{equation}
\nabla_XE_1=-\tau x_2E_3+(\sigma-\tau)x_3E_2, \qquad \nabla_XE_2=\tau
x_1E_3-(\sigma-\tau)x_3E_1. \label{eq-nabla(X,E1)-1}
\end{equation}
Using (\ref{eq:T-angles-detailed}), (\ref{eq:adapted-frame-detailed(e1)}) and
(\ref{eq-nabla(X,E1)-1}), we compute
\begin{equation}
\nabla_XT = \bigl({\rm d}v(X)+(\sigma-\tau)x_3\bigr)e_1 +\tau(x_1\sin v-x_2\cos
v)E_3.\label{eq-nabla(X,E1)-2}
\end{equation}
But
\[
x_3=g(X,E_3), \qquad x_1\sin v-x_2\cos v = -g(X,e_1).
\]
Therefore, from (\ref{eq:C(X)}) and (\ref{eq-nabla(X,E1)-2}), we get
(\ref{eq:nabla-T-detailed}).

Similarly,
\begin{equation}
\nabla_Xe_1 =-\bigl({\rm d}v(X)+(\sigma-\tau)x_3\bigr)T +\tau(x_1\cos v+x_2\sin
v)E_3. \label{eq-nabla(X,E1)-4}
\end{equation}
Since
\[
x_1\cos v+x_2\sin v = g(X,T),
\]
therefore (\ref{eq-nabla(X,E1)-4}) becomes (\ref{eq:nabla-e1-detailed}). Next,
using (\ref{eq:nabla-background-detailed}), we get
\begin{equation}
\nabla_TE_3 = -\tau e_1, \qquad \nabla_{e_1}E_3 = \tau T.
\label{eq-nabla(X,E1)-6}
\end{equation}
From (\ref{eq:xi-angles-detailed}) and (\ref{eq:adapted-frame-detailed(e2)}),
we have
\begin{equation}
\nabla_{e_2}E_3 = -\cos u\,\nabla_TE_3+\sin u\,\nabla_{E_3}E_3 = \tau\cos
u\,e_1, \label{eq-nabla(X,E1)-7}
\end{equation}
\begin{equation}
\nabla_\xi E_3 = \sin u\,\nabla_TE_3+\cos u\,\nabla_{E_3}E_3 =-\tau\sin u\,e_1,
\label{eq-nabla(X,E1)-8}
\end{equation}
which proves (\ref{eq:nabla-E3-special}). Differentiating
(\ref{eq:xi-angles-detailed}) along an arbitrary vector field $X$, we obtain
\begin{equation}
\nabla_X\xi = {\rm d}u(X)(\cos u\,T-\sin u\,E_3) + \sin u\,C(X)e_1-\tau\sin u
g(X,e_1) E_3+\cos u\,\nabla_XE_3. \label{eq-nabla(X,E1)-9}
\end{equation}
Using (\ref{eq:key-sign-detailed}) in (\ref{eq-nabla(X,E1)-9}), we obtain
\begin{equation}
\nabla_X\xi = -{\rm d}u(X)e_2+\sin u\,C(X)e_1 - \tau\sin u g(X,e_1) E_3 +\cos
u\,\nabla_XE_3. \label{eq:nabla-xi-general-proof}
\end{equation}
Substituting $X=e_1,e_2,\xi$ into (\ref{eq:nabla-xi-general-proof}), and using
(\ref{eq-delta}), we respectively obtain (\ref{eq:nabla-e1-xi-detailed}),
(\ref{eq:nabla-e2-xi-detailed}), and (\ref{eq:nabla-xi-xi-detailed}).
$\blacksquare$

Now comparing \eqref{eq:nabla-e1-xi-detailed}, \eqref{eq:nabla-e2-xi-detailed},
and \eqref{eq:nabla-xi-xi-detailed} with \eqref{eq:ABCDPQ}, we have
\[
f_{11}=\sin u\,{\rm d}v(e_1), \qquad f_{12}=-({\rm d}u(e_1)+\tau),
\]
\[
f_{21}=\sin u\,{\rm d}v(e_2)+\tau+\delta\sin^2u, \qquad f_{22}=-{\rm d}u(e_2),
\]
\[
f_{31}=\sin u\bigl({\rm d}v(\xi)+\delta\cos u\bigr), \qquad f_{32}=-{\rm
d}u(\xi).
\]
Substituting the above values into \eqref{eq:master-np-formulas} gives
\begin{align}
\kappa_{\sc np} & = -\frac1{\sqrt2}\Bigl(\sin u\bigl({\rm d}v(\xi)
+\delta\cos u\bigr)+i{\rm d}u(\xi)\Bigr), \label{eq:kappa-detailed}\\
\sigma_{\sc np}&=-\frac12\Bigl(\sin u\,{\rm d}v(e_1)
+{\rm d}u(e_2)-i\bigl(\sin u\,{\rm d}v(e_2)-{\rm d}u(e_1)
+\delta\sin^2u\bigr)\Bigr), \label{eq:sigma-detailed}\\
\rho_{\sc np}&=\frac12\Bigl(\sin u\,{\rm d}v(e_1)
-{\rm d}u(e_2)+i\bigl(\sin u\,{\rm d}v(e_2)+{\rm d}u(e_1)
+2\tau+\delta\sin^2u\bigr)\Bigr).\label{eq:rho-detailed}
\end{align}
Up to this point the computation is entirely local on $U$. We have only fixed a
local unit vector field $\xi$, equivalently a local almost contact metric
structure determined by the oriented orthonormal frame $(e_1,e_2,\xi)$. We will
now impose the trans-Sasakian condition. Since the structure determined by
$(e_1,e_2,\xi)$ is trans-Sasakian if and only if (\ref{eq-tsm-NP}) is true. By
\eqref{eq:kappa-detailed}, \eqref{eq:sigma-detailed}, and
(\ref{eq:rho-detailed}), (\ref{eq-tsm-NP}) is equivalent to
\begin{eqnarray}
&{\rm d}u(\xi)=0,& \label{eq:sys1-detailed}\\
&\sin u\bigl({\rm d}v(\xi)+\delta\cos u\bigr)=0,&\label{eq:sys2-detailed}\\
&{\rm d}u(e_2)+\sin u\,{\rm d}v(e_1)=0,&\label{eq:sys3-detailed}\\
&{\rm d}u(e_1)-\sin u\,{\rm d}v(e_2)-\delta\sin^2u=0,&
\label{eq:sys4-detailed}\\ &\beta = \frac12\Bigl(\sin u\,{\rm d}v(e_1)-{\rm
d}u(e_2)\Bigr), \label{eq:sys4a-detailed}\\ &\alpha = \frac12\bigl(\sin u\,{\rm
d}v(e_2)+{\rm d}u(e_1) +2\tau+\delta\sin^2u\bigr).\label{eq:sys5-detailed}
\end{eqnarray}
Now note that since the horizontal component of $\xi=\sin u\,T+\cos u\,E_3$ is
precisely $\sin u\,T$, the open set on which $\xi$ is not vertical is
\begin{equation}
U=\{\sin u\neq0\}=\{p\in M:\xi(p)\notin\{\pm E_3(p)\}\}. \label{eq-U-sinu}
\end{equation}
On $U$, these equations give
\begin{equation}\label{eq:v-derivatives-detailed}
{\rm d}v(\xi) = -\delta\cos u, \qquad {\rm d}v(e_1)=\frac{\beta}{\sin u},
\qquad {\rm d}v(e_2)=\frac{\alpha-\tau-\delta\sin^2u}{\sin u},
\end{equation}
where
\begin{equation}
\beta = -{\rm d}u(e_2)=\sin u\,{\rm d}v(e_1), \qquad \alpha={\rm
d}u(e_1)+\tau=\sin u\,{\rm
d}v(e_2)+\tau+\delta\sin^2u.\label{eq:alpha-beta-detailed}
\end{equation}
We have therefore reduced the existence of a nonvertical compatible
trans-Sasakian structure vector field to a very explicit local first-order
system on the nonvertical locus $U$.

\begin{lem}\label{lem:bracket-alpha-beta}
Assume that the trans-Sasakian structure equation {\rm (\ref{eq-tsm-NP})} holds
on $U$. Then
\begin{align}
[e_1,\xi]&=\beta e_1-(\alpha+\tau)e_2,\label{eq:bracket-e1xi-detailed}\\
[e_2,\xi]&=(\alpha+\tau)e_1+\beta e_2.\label{eq:bracket-e2xi-detailed}
\end{align}
Moreover,
\begin{align}
\xi(\alpha)&=-2\alpha\beta,\label{eq:xi-alpha-detailed}\\
\xi(\beta)&=\alpha^2-\beta^2-\tau^2.\label{eq:xi-beta-detailed}
\end{align}
\end{lem}

\noindent{\bf Proof.}
Putting $X=\xi$ in \eqref{eq:nabla-e1-detailed}, we get
\begin{equation}
\nabla_\xi e_1=-C(\xi)T+\tau g(\xi,T) E_3. \label{eq-nabla-xi-e1-1}
\end{equation}
Using (\ref{eq-delta}), (\ref{eq:xi-angles-detailed}), (\ref{eq:C(X)}), and
(\ref{eq:v-derivatives-detailed}) in (\ref{eq-nabla-xi-e1-1}), we obtain
\begin{equation}
\nabla_\xi e_1 = -\tau\cos u\,T+\tau\sin u\,E_3  =\tau e_2. \label{eq-nabla-xi-e1-2}
\end{equation}
Now, from (\ref{eq:nabla-e1-xi-detailed}) and (\ref{eq:alpha-beta-detailed}),
we obtain
\begin{equation}
\nabla_{e_1}\xi = \beta e_1 -\alpha e_2. \label{eq-nabla-xi-e1-2a}
\end{equation}
From (\ref{eq-nabla-xi-e1-2}) and (\ref{eq-nabla-xi-e1-2a}), we get
(\ref{eq:bracket-e1xi-detailed}). Now, from \eqref{eq:nabla-e2-xi-detailed} and
\eqref{eq:alpha-beta-detailed}, we have
\begin{equation}
\nabla_{e_2}\xi = \alpha e_1+\beta e_2. \label{eq-nabla-xi-e1-3}
\end{equation}
By \eqref{eq:nabla-xi-xi-detailed}, together with \eqref{eq:sys1-detailed} and
\eqref{eq:sys2-detailed}, we have
\begin{equation}
\nabla_\xi\xi = 0. \label{eq-nabla-xi-e1-4}
\end{equation}
Using (\ref{eq-nabla-xi-e1-2}) and metric compatibility, we get
\begin{equation}
g(\nabla_\xi e_2,e_1) = - g(e_2,\nabla_\xi e_1) = -\tau,
\label{eq-nabla-xi-e1-5}
\end{equation}
\begin{equation}
g(\nabla_\xi e_2,e_2) = 0, \label{eq-nabla-xi-e1-6}
\end{equation}
and
\begin{equation}
g(\nabla_\xi e_2,\xi) = 0. \label{eq-nabla-xi-e1-7}
\end{equation}
Hence
\begin{equation}
\nabla_\xi e_2=-\tau e_1. \label{eq-nabla-xi-e1-8}
\end{equation}
Using (\ref{eq-nabla-xi-e1-3}) and (\ref{eq-nabla-xi-e1-8}), we obtain
(\ref{eq:bracket-e2xi-detailed}). Using (\ref{eq:sys1-detailed}), we have
\begin{equation}
[e_1,\xi](u)=e_1({\rm d}u(\xi))-\xi({\rm d}u(e_1))=-\xi({\rm d}u(e_1)).
\label{eq-nabla-xi-e1-9}
\end{equation}
Using (\ref{eq:alpha-beta-detailed}) in (\ref{eq-nabla-xi-e1-9}), we get
\begin{equation}
[e_1,\xi](u)=-\xi(\alpha). \label{eq-nabla-xi-e1-10}
\end{equation}
On the other hand, using \eqref{eq:bracket-e1xi-detailed} and
(\ref{eq:alpha-beta-detailed}), we obtain
\begin{equation}
[e_1,\xi](u) = \beta {\rm d}u(e_1)-(\alpha+\tau){\rm d}u(e_2) =
\beta(\alpha-\tau) + (\alpha+\tau)\beta = 2\alpha\beta.
\label{eq-nabla-xi-e1-11}
\end{equation}
Thus from (\ref{eq-nabla-xi-e1-10}) and (\ref{eq-nabla-xi-e1-11}), we get
(\ref{eq:xi-alpha-detailed}). Similarly,
\begin{equation}
[e_2,\xi](u) = e_2({\rm d}u(\xi))-\xi({\rm d}u(e_2))=-\xi({\rm
d}u(e_2))=\xi(\beta). \label{eq-nabla-xi-e1-12}
\end{equation}
On the other hand, using \eqref{eq:bracket-e2xi-detailed}, we obtain
\begin{equation}
[e_2,\xi](u) = (\alpha+\tau){\rm d}u(e_1)+\beta {\rm d}u(e_2) =
(\alpha+\tau)(\alpha-\tau)-\beta^2 = \alpha^2-\beta^2-\tau^2.
\label{eq-nabla-xi-e1-13}
\end{equation}
Thus from (\ref{eq-nabla-xi-e1-12}) and (\ref{eq-nabla-xi-e1-13}), we get
(\ref{eq:xi-beta-detailed}). $\blacksquare$

\begin{thm}\label{thm:nonproduct-final}
Let $(M^{3},g)={\Bbb E}(\kappa,\tau)$ with $\tau\neq0$ and $\kappa\neq4\tau^2$.
Then every trans-Sasakian structure compatible with $g$ is vertical. More
precisely, on each connected component,
\[
\xi=\pm E_3.
\]
Hence no proper trans-Sasakian structures are compatible with the homogeneous
metric in the genuine non-product, non-space-form cases of type ${\Bbb
E}(\kappa,\tau)$; the only compatible structures are the canonical vertical
$\alpha$-Sasakian structures.
\end{thm}

\noindent{\bf Proof.} It is enough to show that the nonvertical locus $U$
defined by (\ref{eq-U-sinu}) is empty. From (\ref{eq:v-derivatives-detailed}),
(\ref{eq:alpha-beta-detailed}) and (\ref{eq:bracket-e1xi-detailed}), we obtain
\begin{equation}
[e_1,\xi](v) = \frac{\beta^2-(\alpha+\tau) (\alpha-\tau-\delta\sin^2u)}{\sin
u}. \label{eq-[e1,xi]v-th-1}
\end{equation}
Using (\ref{eq:sys1-detailed}), (\ref{eq:v-derivatives-detailed}), and
\eqref{eq:xi-beta-detailed}, we get
\begin{equation}
e_1({\rm d}v(\xi)) =e_1(-\delta\cos u) = \delta\sin u\,(\alpha-\tau),
\label{eq-[e1,xi]v-th-2}
\end{equation}
\begin{equation}
\xi({\rm d}v(e_1)) = \xi\left(\frac{\beta}{\sin u}\right)
=\frac{\alpha^2-\beta^2-\tau^2}{\sin u}. \label{eq-[e1,xi]v-th-3}
\end{equation}
Hence from (\ref{eq-[e1,xi]v-th-2}) and (\ref{eq-[e1,xi]v-th-3}), we have
\begin{equation}
[e_1,\xi](v) = \delta\sin u\,(\alpha-\tau) -\frac{\alpha^2-\beta^2-\tau^2}{\sin
u}. \label{eq-[e1,xi]v-th-4}
\end{equation}
Comparing (\ref{eq-[e1,xi]v-th-1}) and (\ref{eq-[e1,xi]v-th-4}), we get
\begin{equation}\label{eq:nonproduct-contradiction}
2\tau\delta\sin^2u=0.
\end{equation}
Since
\[
\tau\delta = \tau(\sigma-2\tau) = \frac{\kappa-4\tau^2}{2},
\]
our assumptions $\tau\neq0$ and $\kappa\neq4\tau^2$ imply $\tau\delta\neq0$.
Therefore \eqref{eq:nonproduct-contradiction} forces $\sin u=0$ on $U$,
contradicting the definition of $U$. Hence $U=\varnothing$, and so $\xi$ is
vertical everywhere. Since $\xi$ is a unit vector field, this means
\[
\xi=\pm E_3
\]
on each connected component. $\blacksquare$

\begin{rem}
The condition $\kappa\neq4\tau^2$ is the usual non-space-form condition for the
genuine ${\Bbb E}(\kappa,\tau)$ geometries. When $\kappa=4\tau^2$, the argument
above degenerates, as expected in the space-form case.
\end{rem}

\subsection{The product case $\tau=0$}
Let
\[
M^3= M^2_{\kappa}\times {\Bbb R}
\]
with the product metric, where $M^2_{\kappa}$ is the simply connected surface
of constant curvature $\kappa$. Choose a local positively oriented orthonormal
frame $(E_1,E_2)$ on $M^2_{\kappa}$ and let $E_3=\partial_t$.  Write
\begin{equation}
\nabla_XE_1=\omega(X)E_2, \qquad \nabla_XE_2=-\omega(X)E_1,\qquad
\nabla_XE_3=0, \label{eq-nabla(X,E1,E2,E3)-pro}
\end{equation}
where $\omega$ is the Levi-Civita connection $1$-form of $M^2_{\kappa}$. Then
\begin{equation}
[E_1,E_2]=-\omega(E_1)E_1-\omega(E_2)E_2, \qquad [E_2,E_3]=0,\qquad
[E_3,E_1]=0,\label{[X,E1,E2,E3]-pro}
\end{equation}
and the curvature of the surface is encoded by
\begin{equation}
d\omega=-\kappa\,\theta^1\wedge\theta^2, \label{eq-kappa-theta12-pro}
\end{equation}
where $(\theta^1,\theta^2)$ is the dual co-frame to the orthonormal frame
$(E_1,E_2)$. Let now $\xi$ denote the structure vector field of some compatible
trans-Sasakian structure. We now consider again on the nonvertical locus
\[
U=\{p:\xi(p)\notin\{\pm E_3(p)\}\}.
\]
and write analogously to \eqref{eq:xi-angles-detailed} and
\eqref{eq:T-angles-detailed}
\begin{equation}
\xi=\sin u\,T+\cos u\,E_3,  \label{eq-xi-pro}
\end{equation}
\begin{equation}
T=\frac{\xi^{\mathrm{hor}}}{\Vert\xi^{\mathrm{hor}}\Vert}=\cos v\,E_1+\sin
v\,E_2, \label{eq-T-pro}
\end{equation}
and complete $\xi$ to the adapted frame
\begin{equation}
e_1=-\sin v\,E_1+\cos v\,E_2,  \label{e1-pro}
\end{equation}
\begin{equation}
e_2=-\cos u\,T+\sin u\,E_3.  \label{e2-pro}
\end{equation}

\begin{lem}\label{lem:product-covariant-derivatives}
With the above notation, define
\begin{equation}
B(X):= {\rm d}v(X)+\omega(X) \label{eq-B(X)-prod}
\end{equation}
for any vector field $X$. Then,
\begin{equation}\label{eq:nabla-T-product-detailed}
\nabla_XT=B(X)e_1, \qquad \nabla_Xe_1=-B(X)T.
\end{equation}
Moreover,
\begin{equation}\label{eq:nabla-xi-product-detailed}
\nabla_X\xi=\sin u\,B(X)e_1-{\rm d}u(X)e_2.
\end{equation}
\end{lem}

\noindent{\bf Proof.} From (\ref{eq-nabla(X,E1,E2,E3)-pro}), (\ref{eq-T-pro})
and (\ref{e1-pro}), we obtain
\begin{equation*}
\nabla_XT=({\rm d}v(X)+\omega(X))e_1=B(X)e_1,
\end{equation*}
\begin{equation*}
\nabla_Xe_1=-({\rm d}v(X)+\omega(X))T=-B(X)T,
\end{equation*}
which proves (\ref{eq:nabla-T-product-detailed}).
Differentiating (\ref{eq-xi-pro}), we get
\begin{equation}
\nabla_X\xi = \cos u\,{\rm d}u(X)T+\sin u\,B(X)e_1-\sin u\,{\rm d}u(X)E_3.
\label{eq-nabla(X,xi)-pro}
\end{equation}
Using (\ref{e2-pro}) in (\ref{eq-nabla(X,xi)-pro}), we get
(\ref{eq:nabla-xi-product-detailed}). $\blacksquare$

Putting $X=e_1,e_2,\xi$ in (\ref{eq:nabla-xi-product-detailed}), we
respectively obtain
\begin{equation}
\nabla_{e_1}\xi=\sin u\,B(e_1)e_1- {\rm d}u(e_1)e_2,
\label{eq:nabla-e1xi-product-detailed}
\end{equation}
\begin{equation}
\nabla_{e_2}\xi=\sin u\,B(e_2)e_1- {\rm d}u(e_2)e_2,
\label{eq:nabla-e2xi-product-detailed}
\end{equation}
\begin{equation}
\nabla_\xi \xi=\sin u\,B(\xi)e_1- {\rm d}u(\xi)e_2.
\label{eq:nabla-xixi-product-detailed}
\end{equation}
Comparing (\ref{eq:nabla-e1xi-product-detailed}),
(\ref{eq:nabla-e2xi-product-detailed}), and
(\ref{eq:nabla-xixi-product-detailed}) with (\ref{eq:ABCDPQ}), we get
\[
f_{11} = \sin u\,B(e_1), \qquad f_{12} = -{\rm d}u(e_1),
\]
\[
f_{21}=\sin u\,B(e_2), \qquad f_{22} = -{\rm d}u(e_2),
\]
\[
f_{31}=\sin u\,B(\xi), \qquad f_{32} = -{\rm d}u(\xi).
\]
Substituting the above values into \eqref{eq:master-np-formulas} gives
\begin{align*}
\kappa_{\sc np} &=-\frac1{\sqrt2}\bigl(\sin u\,B(\xi)+i{\rm d}u(\xi)\bigr),\\
\sigma_{\sc np} &=-\frac12\Bigl(\sin u\,B(e_1)+{\rm d}u(e_2)-i\bigl(\sin
u\,B(e_2)-{\rm d}u(e_1)\bigr)\Bigr),\\ \rho_{\sc np}&=\frac12\Bigl(\sin
u\,B(e_1)-{\rm d}u(e_2)+i\bigl(\sin u\,B(e_2)+{\rm d}u(e_1)\bigr)\Bigr).
\end{align*}
Hence the trans-Sasakian condition (\ref{eq-tsm-NP}) is equivalent to
\begin{eqnarray}
&{\rm d}u(\xi)=0,& \label{eq:prod1-detailed}\\
&\sin u\,B(\xi)=0,& \label{eq:prod2-detailed}\\
&{\rm d}u(e_2)+\sin u\,B(e_1)=0,& \label{eq:prod3-detailed}\\
&{\rm d}u(e_1)-\sin u\,B(e_2)=0,&\label{eq:prod4-detailed}\\
&\beta = \frac12\Bigl(\sin u\,B(e_1)-{\rm
d}u(e_2)\Bigr),\label{eq:prod5-detailed}\\ &\alpha = \frac12\bigl(\sin
u\,B(e_2)+{\rm d}u(e_1)\bigr).\label{eq:prod6-detailed}
\end{eqnarray}
Now note that since the horizontal component of $\xi=\sin u\,T+\cos u\,E_3$ is
precisely $\sin u\,T$, the open set on which $\xi$ is not vertical is
\begin{equation}
U=\{\sin u\neq0\}=\{p:\xi(p)\notin\{\pm E_3(p)\}\}. \label{eq-U-sinu-prod}
\end{equation}
On $U$, these equations give
\begin{equation}\label{eq:alpha-beta-product-detailed}
\alpha={\rm d}u(e_1)=\sin u\,B(e_2), \qquad \beta=-{\rm d}u(e_2)=\sin
u\,B(e_1), \qquad B(\xi)=0.
\end{equation}

We can now finish the product case. The role played by the constant bracket
terms in the non-product case is now played by the curvature of the horizontal
connection form $B$ defined by (\ref{eq-B(X)-prod}).

\begin{thm}\label{thm:product-final}
Let
\[
M^3=M^2_{\kappa}\times {\Bbb R}
\]
with the product metric, and assume $\kappa\neq0$. Then every trans-Sasakian
structure compatible with the product metric is vertical. More precisely, on
each connected component,
\[
\xi=\pm E_3.
\]
Hence every such structure is cosymplectic.
\end{thm}

\noindent{\bf Proof.} As in the non-product case it is enough to show that the
nonvertical locus $U$ defined by (\ref{eq-U-sinu-prod}) is empty.
From \eqref{eq:nabla-xi-product-detailed}, and
\eqref{eq:alpha-beta-product-detailed}, we have
\begin{equation}
\nabla_{e_1}\xi=\beta e_1-\alpha e_2, \qquad \nabla_{e_2}\xi=\alpha e_1+\beta
e_2. \label{eq-nabla(e1,xi)-prod}
\end{equation}
Also, by \eqref{eq:nabla-T-product-detailed},
\begin{equation}
\nabla_\xi e_1=-B(\xi)T=0 \label{eq-nabla(xi,e1)-prod}
\end{equation}
on $U$. Since $\nabla_\xi\xi=0$, orthonormality of $(e_1,e_2,\xi)$ gives
\begin{equation}
\nabla_\xi e_2=0. \label{eq-nabla(xi,e2)}
\end{equation}
Therefore from (\ref{eq-nabla(e1,xi)-prod}), (\ref{eq-nabla(xi,e1)-prod}) and
(\ref{eq-nabla(xi,e2)}), we get
\begin{equation}
[e_1,\xi]=\beta e_1-\alpha e_2, \qquad [e_2,\xi]=\alpha e_1+\beta e_2.
\label{eq:product-brackets}
\end{equation}
Hence
\begin{equation}
[e_1,\xi](u) = \beta {\rm d}u(e_1)-\alpha {\rm d}u(e_2) = 2\alpha\beta.
\label{eq:product-brackets-1}
\end{equation}
But on the other hand, using (\ref{eq:prod1-detailed}), we have
\begin{equation}
[e_1,\xi](u) = e_1({\rm d}u(\xi))-\xi({\rm d}u(e_1)) = -{\rm d}\alpha(\xi).
\label{eq:product-brackets-2}
\end{equation}
By comparing (\ref{eq:product-brackets-1}) and (\ref{eq:product-brackets-2}),
we get \begin{equation}\label{eq:product-xi-alpha} {\rm d}\alpha(\xi) =
-2\alpha\beta.
\end{equation}
Similarly,
\[
[e_2,\xi](u)=-\xi({\rm d}u(e_2)) = {\rm d}\beta(\xi),
\]
while
\[
[e_2,\xi](u) = \alpha {\rm d}u(e_1)+\beta {\rm d}u(e_2) = \alpha^2-\beta^2.
\]
Hence
\begin{equation}\label{eq:product-xi-beta}
{\rm d}\beta(\xi) = \alpha^2-\beta^2.
\end{equation}
Using (\ref{eq:alpha-beta-product-detailed}), we compute
\begin{equation}
2dB(e_1,\xi) =e_1(B(\xi))-\xi(B(e_1))-B([e_1,\xi]) = - \frac{{\rm
d}\beta(\xi)}{\sin u} -\frac{\beta^2-\alpha^2}{\sin u}.
\label{eq-dB-domega(e1,xi)}
\end{equation}
Using \eqref{eq:product-xi-beta} in (\ref{eq-dB-domega(e1,xi)}), we obtain
\begin{equation}
dB(e_1,\xi)=0. \label{dB(e1,xi)}
\end{equation}
On the other hand, from (\ref{eq-B(X)-prod}), we have
\begin{equation}
dB=d\omega=-\kappa\,\theta^1\wedge\theta^2. \label{eq-dB-domega}
\end{equation}
Also,
\begin{equation}
dB(e_1,\xi) = -\kappa\,\theta^1\wedge\theta^2(e_1,\xi). \label{dB(e1,xi)-1}
\end{equation}

The form $\theta^1\wedge\theta^2$ is the area form of the horizontal factor,
with respect to the oriented frame $(E_1,E_2)$. It therefore only sees the
horizontal part of its arguments. Since $\xi=\sin u\,T+\cos u\,E_3$ and $E_3$
is vertical, we have
\begin{equation}
\theta^1\wedge\theta^2(e_1,\xi) = \sin u\,\theta^1\wedge\theta^2(e_1,T).
\label{eq-theta1-wedge-theta2(e1,xi)}
\end{equation}
Now, from (\ref{eq-T-pro}) and (\ref{e1-pro}), we have
\begin{equation}\label{eq-theta1-wedge-theta2(e1,T)}
\theta^1\wedge\theta^2(e_1,T) = \det
\begin{pmatrix}
-\sin v & \cos v\\
\cos v & \sin v
\end{pmatrix}
=-1.
\end{equation}
Hence from (\ref{eq-theta1-wedge-theta2(e1,xi)}) and
(\ref{eq-theta1-wedge-theta2(e1,T)}), we get
\begin{equation}
\theta^1\wedge\theta^2(e_1,\xi)=-\sin u.
\label{eq-theta1-wedge-theta2(e1,xi)-1}
\end{equation}
In view of (\ref{eq-theta1-wedge-theta2(e1,xi)-1}), (\ref{dB(e1,xi)-1}) becomes
\begin{equation}
dB(e_1,\xi)=\kappa\sin u. \label{dB(e1,xi)-2}
\end{equation}
Comparing (\ref{dB(e1,xi)}) and (\ref{dB(e1,xi)-2}), we get
\begin{equation}
\kappa\sin u=0.
\end{equation}
Since $\kappa\neq0$, this forces $\sin u=0$, contradicting the definition of
$U$. Hence $U=\varnothing$, and so $\xi$ is vertical everywhere. Since $\xi$ is
a unit vector field, this means
\[
\xi=\pm E_3
\]
on each connected component. Finally, for the vertical field $E_3=\partial_t$
in the product metric, one has $\nabla E_3=0$. Thus the corresponding
compatible trans-Sasakian structure has
\[
\alpha=0,\qquad \beta=0,
\]
and is cosymplectic.
$\blacksquare$

The novelty of the result is that the rigidity is not a consequence of assuming
the almost contact structure to be homogeneous. Indeed, the structure vector
field $\xi$ is allowed to be an arbitrary smooth unit vector field. The proof
shows that the trans-Sasakian equations themselves, together with the curvature
of the canonical fibration of ${\Bbb E}(\kappa,\tau)$, force the non-vertical
component of $\xi$ to vanish.

\begin{rem}
Note that the flat geometry ${\mathbb E}(0,0)=({\Bbb R}^3,g_0)$ behaves
differently from the non-flat ${\mathbb E}(\kappa,\tau)$-geometries.  In the
latter cases, the Newman--Penrose equations force the structure vector field to
coincide with the vertical field of the canonical fibration.  In the flat case,
this rigidity fails locally.

Indeed, on punctured Euclidean space ${\Bbb R}^3\setminus\{0\}$, the radial
unit vector field
\[
\xi=\frac{x}{r}\partial_x+\frac{y}{r}\partial_y+\frac{z}{r}\partial_z, \qquad
r=\sqrt{x^2+y^2+z^2},
\]
satisfies
\[
\nabla_X\xi=\frac1r\bigl(X-\eta(X)\xi\bigr),
\]
and therefore
\[
\kappa_{\sc np}=0,\qquad
\sigma_{\sc np}=0,\qquad
\rho_{\sc np}=\frac1r.
\]
Hence the associated almost contact metric structure is trans-Sasakian of type
$(\alpha,\beta)=\left(0,\frac1r\right)$. Thus the flat exception is genuine
locally: there exist non-cosymplectic trans-Sasakian structures compatible with
the flat metric. Globally, however, the Newman--Penrose equations become much
more rigid. In the flat case, the Sachs equation reduces to
\[
\xi(\rho_{\sc np})=-\rho_{\sc np}^2.
\]
For a complete regular congruence on all of ${\Bbb R}^3$, any non-zero solution
develops a singularity in finite parameter time along the geodesics of $\xi$.
Consequently, a smooth globally defined trans-Sasakian structure on Euclidean
space must satisfy $\rho_{\sc np}=0$. Equivalently, $\alpha=0$, and $\beta=0,$
so the structure is cosymplectic.

Thus the flat model behaves differently only at the local level. On the
complete Euclidean space, the only regular global trans-Sasakian structures are
the parallel ones.  This agrees with the Baird--Wood classification of regular
conformal foliations by geodesics in $3$-dimensional Euclidean space
{\rm\cite[Baird and Wood 1991]{Baird-Wood-91-JAMS}}, but here it appears
naturally from the Newman--Penrose equations themselves.
\end{rem}

\begin{rem}
It is worth mentioning that among the Thurston geometries, the Sol geometry
behaves in a fundamentally different way from the other cases studied above.
Indeed, consider the standard Sol metric
\[
g=e^{2z}dx^2+e^{-2z}dy^2+dz^2
\]
with orthonormal frame
\[
E_1=e^{-z}\partial_x,\qquad E_2=e^z\partial_y,\qquad E_3=\partial_z .
\]
For this metric the Ricci tensor is
\[
S=-2\,\theta^3\otimes\theta^3,
\]
where $\theta^3$ is dual to $E_3$.

Suppose that a trans-Sasakian structure exists on some open set. Then
$\kappa_{\sc np}=0$, $\sigma_{\sc np}=0$. Substituting these conditions into
the first generalized Sachs equation gives
\[
S(\partial,\partial)=0.
\]
Writing
\[
p=g(E_3,e_1),\qquad q=g(E_3,e_2),
\]
we have
\[
g(E_3,\partial)=\frac1{\sqrt2}(p-iq),
\]
and hence
\[
S(\partial,\partial)=-2g(E_3,\partial)^2=-(p-iq)^2.
\]
Thus $p=q=0$, so $E_3$ is orthogonal to $e_1,e_2$. Therefore $E_3$ is parallel
to $\xi$, and since both vector fields have unit length, $\xi=\pm E_3$ on each
connected component.

It remains only to check the vertical field.  If $\xi=E_3$, then with
\[
\partial=\frac1{\sqrt2}(E_1-iE_2)
\]
and using
\[
\nabla_{E_1}E_3=E_1,\qquad \nabla_{E_2}E_3=-E_2,
\]
we get
\[
\nabla_\partial\xi = \frac1{\sqrt2}(E_1+iE_2) = \overline{\partial}.
\]
Consequently
\[
\sigma_{\sc np} = -g(\partial,\nabla_\partial\xi) =
-g(\partial,\overline{\partial}) = -1,
\]
contradicting $\sigma_{\sc np}=0$. The case $\xi=-E_3$ is the same, up to sign.
Hence the standard Sol metric admits no local compatible trans-Sasakian
structure.

What is striking in this argument is that the obstruction appears directly at
the level of the Newman--Penrose equations. The combination of the shear-free
condition with a single Sachs equation immediately detects the incompatibility
between the Sol Ricci tensor and the trans-Sasakian equations. In particular,
the Newman--Penrose formalism turns the problem into a very short rigidity
argument expressed entirely through scalar quantities associated with the
congruence.
\end{rem}

\begin{rem}
The remaining space-form cases, which complete the list of Thurston geometries,
are not treated separately here, since they do not exhibit the rigidity
phenomenon considered in this section. In the family ${\Bbb E}(\kappa,\tau)$,
the round sphere corresponds to $\kappa=4\tau^2$; in this case the Ricci tensor
is isotropic and the Newman--Penrose equations no longer force the structure
vector field to be vertical. Hyperbolic space ${\Bbb H}^3$ also falls outside
the vertical-fibration mechanism used above.  A Newman--Penrose study of these
space-form cases, and more generally of warped product geometries, would be
natural: there the method would not primarily give vertical rigidity, but
rather an efficient way to analyse shear-free geodesic congruences and the
corresponding trans-Sasakian structures.
\end{rem}
\vspace{3mm}

\noindent {\bf Acknowledgement.} The first author gratefully acknowledges the
Department of Science and Technology (DST), Government of India, for granting
the INSPIRE Fellowship [DST/INSPIRE/03/2022/002296], inspire code IF210729. The
third author is thankful to the Banaras Hindu University for the ``Incentive
Grant" under the IoE Scheme (BHU), Ref.\ No.\ R/Dev/D/IoE/Incentive
(Phase-IV)/2024-25/82489.

\end{document}